\documentclass{pspum-l}
\usepackage{pdfsync}
\usepackage{amsmath,amsthm,amsfonts,latexsym,amscd,amssymb,enumerate,mathrsfs}
\usepackage{graphicx}
\usepackage{comment}
\usepackage{extarrows}
\usepackage[backrefs,?,numeric]{amsrefs}
\usepackage[all]{xy}
\usepackage[pdftex,pdftitle={Numeric rho invariants}, backref,pdftex]{hyperref}

\swapnumbers
\theoremstyle{plain}
\newtheorem{theorem}{Theorem}[section]
\newtheorem{lemma}[theorem]{Lemma}

\newtheorem{proposition}[theorem]{Proposition}

\newtheorem{assumption}[theorem]{Assumption}

\theoremstyle{definition}
\newtheorem{definition}[theorem]{Definition}
\newtheorem{example}[theorem]{Example}

\newtheorem{set-up}[theorem]{Geometric set-up}
\newtheorem{remark}[theorem]{Remark}

\DeclareMathOperator{\cyl}{cyl}

 \DeclareMathOperator{\tr}{tr}

\DeclareMathOperator{\Ind}{Ind}

\newcommand{\forget}[1]{}

\def  \nuint {\raise10pt\hbox{$\nu$}\kern-6pt\int}

\newcommand\Tr{\operatorname{Tr}}

\def \Sp {{\cal S}}

\def\Id{{\rm Id}}

\newcommand\ch{\operatorname{ch}}

\newcommand\Di{D\kern-6pt/}
\newcommand\cDi{{\mathcal D}\kern-6pt/}
\newcommand\spi{S\kern-6pt/}
\newcommand \cspi{\Sp\kern-6pt/}

\newcommand\CC{\mathbb C}

\def \cal {\mathcal}

\newcommand\RR{\mathbb R}
\newcommand\ZZ{\mathbb Z}

\newcommand\pa{\partial}
\newcommand\Ker{\operatorname{Ker}}

 \usepackage{xcolor}
\definecolor{darkgreen}{cmyk}{1,0,1,.2}
\definecolor{m}{rgb}{1,0.1,1}

\newcommand{\bint}{\sideset{^{\mathrm{b}\!\!\!}}{_{Y}}\int}

{\catcode`@=11\global\let\c@equation=\c@theorem}

\allowdisplaybreaks[2]

\begin{document}

\title[Primary and secondary invariants]{Primary and secondary invariants of Dirac operators on $G$-proper manifolds}


\author{Paolo Piazza}
\address{Dipartimento di Matematica, Sapienza Universit\`{a} di Roma, I-00185 Roma, Italy}
\email{piazza@mat.uniroma1.it}
\thanks{The first author was supported in part by {\it Ministero Universit\`a e Ricerca}, through the PRIN {\it Spazi di moduli e teoria di Lie}.}

\author{Xiang Tang}
\address{Department of Mathematics and Statistics, Washington University, St. Louis, MO, 63130, U.S.A.}
\email{xtang@math.wustl.edu}
\thanks{The second author was supported in part by NSF Grants DMS-1800666 and DMS-1952551.}


\dedicatory{For the 40th birthday of cyclic cohomology}


\begin{abstract}
In this article, we survey the recent constructions of cyclic cocycles on the Harish-Chandra Schwartz algebra of a connected real reductive Lie group $G$ and their applications to higher index theory for proper cocompact $G$-actions. 
\end{abstract}

\maketitle

\tableofcontents

\section{Introduction}

Let $G$ be a connected reductive linear Lie group. The study of index theory on proper cocompact $G$-manifolds goes back to the 70s, in work of Atiyah \cite{Atiyah:L2}, Atiyah-Schmid \cite{Atiyah-Schmid}, and Connes-Moscovici \cite{Connes-Moscovici:L2}. It is deeply connected to geometric analysis, noncommutative geometry, and representation theory. In September 2021, in the Conference ``Cyclic Cohomology at 40", organized by the Fields Institute, the authors reported 
on the recent studies of cyclic cocycles on the Harish-Chandra Schwartz algebra $\mathcal{C}(G)$ and their applications to higher index theory. In this paper, we expand our two talks and provide a more complete account of related topics.  The article is organized as follows. 

In Section \ref{sec:invariant elliptic}, we introduce the geometric set up for proper cocompact $G$-manifolds and define the $C^*$-algebraic index of an invariant elliptic operator on a proper cocompact $G$-manifold. We also explain the connection to the Connes-Kasparov isomorphism conjecture/theorem.  

In Section \ref{sec:cyclic cocycles}, after briefly introducing the definition of cyclic cohomology, we present two approaches to construct cyclic cocycles on the Harish-Chandra Schwartz algebra $\mathcal{C}(G)$. One source of cyclic cocycles is the differentiable group cohomology of $G$; the other source is the 
orbital integral and its generalization  to higher cyclic cocycles.

In Section \ref{sec:higher indices}, we study higher index theory for proper cocompact $G$-manifolds that  have no boundaries. More precisely, we discuss two higher index theorems corresponding to the above two types of cyclic cocycles on $\mathcal{C}(G)$. We also explain these higher index theorems in the special case of $X=G/K$, with their connections to representation theory. 

In Section \ref{sec:boundary}, we extend the study of higher index theory to proper cocompact $G$-manifolds with boundary. We introduce the geometric set up for proper cocompact $G$-manifold with boundary and adapt the Melrose $b$-calculus to investigate the higher index on these manifolds.   

In Section \ref{sec:aps index}, we introduce the framework of relative cyclic cohomology and apply it to present two higher Atiyah-Patodi-Singer
 index theorems associated to the cyclic cocycles introduced in Section \ref{sec:cyclic cocycles}. Higher rho invariants are introduced as spectral invariants for proper cocompact $G$-manifolds without boundary.

In Section \ref{sec:application}, we discuss applications of these results   to interesting  problems in topology and geometry; in particular
we introduce higher genera for $G$-proper manifolds without boundary and explain their stability properties and their cut-and-paste
behaviour; we also discuss bordism invariance of the rho numbers we have introduced. \\

\noindent{\bf Acknowledgments.} 
We take this opportunity to thank Alain Connes for his foundational role in the development of cyclic cohomology. Ideas and questions by Connes and his collaborators have inspired many developments that we are going to survey in this article. We also  thank the organizers of the conference, ``Cyclic Cohomology at 40", Alain Connes, Katia Consani,
Masoud Khalkali and Henri Moscovici, for organizing such an interesting online event, providing the participants great opportunities to meet and exchange ideas.

We are very happy  to thank Markus Pflaum, Hessel Posthuma and Yanli Song for inspiring discussions.

\section{Invariant elliptic operators on proper cocompact manifolds} 
\label{sec:invariant elliptic}

Let $G$ be a connected reductive linear Lie group with a maximal compact subgroup $K$. We introduce in this section the index of a $G$-equivariant elliptic operator on a proper cocompact $G$-manifold. 

\subsection{Geometry of Proper Cocompact Manifolds} 
\begin{definition} A smooth manifold $X$ is called a $G$-proper manifold if $X$ is equipped with a proper $G$-action, that is, the associated map 
\[
G\times X\to X\times X,\quad (g,x)\mapsto (x,gx),\quad g\in G,x\in X,
\]
is a proper map. This implies that the stabilizer groups $G_x$ of all points $x\in X$ are compact and that the quotient space $X\slash G$ is Hausdorff. The action is 
said to be {\em cocompact} if the quotient $X/G$ is compact. Finally, a cocompact $G$-proper manifold can always be endowed with a cut-off
function $c_X$, a smooth compactly supported function on $X$  satisfying
\[
\int_G c_X (g^{-1}x)dg=1,\quad\mbox{for all}~x\in X
\]
\end{definition}

The following theorem was proved by Abels \cite{abels:slice} for proper cocompact $G$-manifolds, and will be used crucially in our study of  index theory.
\begin{theorem}\label{thm:slice}Let $X$ be a proper cocompact $G$-manifold (with or without boundary). There is a compact submanifold $Z$ (with or without boundary) which is equipped with a $K$-action such that 
\[
X\cong G\times _K Z. 
\]
\end{theorem}

\begin{assumption}We assume in this paper  that the manifolds $G/K$, $X$ and $Z$  are even dimensional. 
\end{assumption}

Choose a $K$-invariant inner product on the Lie algebra $\mathfrak{g}$ of $G$. 
We then have an orthogonal decomposition $\mathfrak{g}=\mathfrak{k}\oplus \mathfrak{p}$ where $\mathfrak{k}$ is the Lie algebra of $K$ and $\mathfrak{p}$ its orthogonal complement. Accordingly, we have an isomorphism
\begin{equation}\label{eq:TY}
TY\cong G\times _K (\mathfrak{p}\oplus TZ),
\end{equation}
where in the above decomposition we have abused notation  and employed $\mathfrak{p}$ for the trivial vector bundle $\mathfrak{p}\times Z\to Z$.  

\begin{definition}\label{def:metric} We say that the $G$-invariant metric $\mathbf{h}$ on $X$ is slice-compatible, if it is obtained by a $K$-invariant metric on $Z$ and a $K$-invariant metric on $\mathfrak{p}$ via Equation (\ref{eq:TY}). 
\end{definition}

We assume that that adjoint representation $\operatorname{Ad}: K\to SO(\mathfrak{p})$ admits a lift $\widetilde{\operatorname{Ad}}: K\to \operatorname{Spin}(\mathfrak{p})$. By the following exact sequence of vector bundles and the two out of three lemma for $\operatorname{Spin}^c$-structures, 
\[
0 \to G \times_K \mathfrak{p} \to TX \to G \times_K TZ \to 0, 
\] 
we see that a $K$-invariant $\operatorname{Spin}^c$-structure on $Z$
induces a $G$-invariant  $\operatorname{Spin}^c$-structure on $X$. 

\begin{definition}\label{def:spinc}
We shall say that a $G$-invariant Spin$^c$-structure on $X$ is slice-compatible if it is associated  to a K-invariant metric and a K-invariant Spin$^c$-structure on the slice  in the above way.
\end{definition}

We consider a $G$-equivariant twisted spinor bundle $E$ on $X$, i.e. $E=\mathcal{S}\otimes W$ (for the spinor bundle $\mathcal{S}$ on $X$ and a $G$-equivariant Hermitian vector bundle $W$).
We shall consider $E$ as arising from a $K$-invariant twisted spinor bundle $E_Z$ on $Z$, defined by 
the slice-compatible Spin$^c$-structure on $Z$ and by  an auxiliary  $K$-equivariant vector bundle on $Z$. 
Then $E$ is equipped with a $G$-equivariant $\text{Cliff}(TX)$-module structure. A \emph{Clifford connection} $\nabla^E$ on $E$ is a connection on $E$ satisfying
\[
\left[\nabla^{E}_V, \mathbf{c}(V') \right] = \mathbf{c}(\nabla^{TX}_VV'), \quad V, V' \in C^\infty(X, TX)
\]
where $\mathbf{c}$ is the Clifford action and $\nabla^{TX}$ is the Levi-Civita connection on $X$. The Dirac operator associated to the Clifford connection is given by the following composition
\[
D \colon C^\infty(X, E)\xlongrightarrow{\nabla^{E}} C^\infty(X, T^*X \otimes E)  \cong C^\infty(X, TX\otimes E) \xlongrightarrow{\mathbf{c}}C^\infty(X, E). 
\]

The vector bundle $E$ on $X$ is defined by a $K$-equivariant vector bundle $E_{Z}$ on $Z$, i.e. 
\[
E \cong G \times_K( S_\mathfrak{p} \otimes E_{Z}), 
\]
and $E_{Z}$ admits a $K$-equivariant $\text{Cliff}(TZ)$-module structure. Accordingly, we can decompose 
\[
L^2(Y, E) \cong \left[L^2(G) \otimes S_\mathfrak{p} \otimes L^2(Z, E_{Z}) \right]^K. 
\]
By the assumption that the metric is slice-compatible, we introduce a new split Dirac operator $D_{\rm split}$  by the following formula 
\begin{equation}
\label{dec dirac}
D_{\rm split} = D_{G,K} \hat{\otimes} 1 + 1\hat{\otimes} D_{Z},  
\end{equation}
where $D_{G,K}$ is the Spin$^c$-Dirac operator on $(L^2(G)\otimes S_\mathfrak{p})^K$, $D_{Z}$ is a $K$-equivariant Dirac operator on $E_{Z}$, and $\hat{\otimes}$ means the graded tensor product. We make the following observation, 
\[
[D_{G,K} \hat{\otimes} 1,1\hat{\otimes} D_{Z}] = 0, \quad D_{\rm split}^2 = D_{G,K}^2 + D_{Z}^2. 
\]
There was a confusion in literature that the split Dirac operator $D_{\rm split}$ is identical to $D$, e.g. \cite[Section 3]{HS-decomposition}. It is not hard to see that the two Dirac operators $D$ and $D_{\rm split}$ have identical principal symbols when the metric is slice compatible. Accordingly, the two operators defines the same $G$-equivariant $K$-homology class on $X$ and therefore the same index element $\Ind(D)$ and $\Ind(D_{\rm split})$ in $K_0(C^*_r(G))$. However, in general, the two operators are not identical. And therefore, the spectral invariants associated to $D$ and $D_{\rm split}$ are a priori not same. We refer to \cite{PPST} for a more detailed discussion. 

\subsection{Roe's $C^*$-algebra}

Roe's $C^*$-algebra \cite{roe:open} for a complete proper metric space $X$ is a powerful tool to study higher index theory.  

More precisely, for $r>0$, let $C_{fp}(X\times X)$ be the space of bounded continuous functions on $X\times X$ satisfying $f(x,y)=0$ for $d(x,y)> r$ for some $r>0$. Elements of $C_{fp}(X\times X)$ naturally define kernels of bounded linear operators on $L^2(X)$. Define $C^*(X)$ to be the completion of $C_{fp}(X\times X)$ with respect to the operator norm on $L^2(X)$. Composition of kernels makes $C^*(X)$ into a $C^*$-algebra. For a proper cocompact $G$-manifold $X$ equipped with a $G$-invariant proper complete metric, we consider $C_{fp}^G(X\times X)$ which is a subspace of $C_{fp}(X\times X)$ consisting of $G$-invariant bounded continuous functions. And we denote the completion of $C_{fp}^G(X\times X)$ with respect to the operator norm by $C^*_G(X)$, which is a $C^*$-algebra with respect to operator composition. 

In this paper, we will work with smooth subalgebra of $C^*_G(X)$ using the slice theorem (Theorem \ref{thm:slice}), i.e. $X\cong G\times _K Z$. Under the above diffeomorphism, $C_{fp}^G(X\times X)$ 
can be identified with $\mathcal{A}_G(X)$ which is defined as
\[
\mathcal{A}_G(X):= \big( C_c(G) \hat{\otimes} C(S\times S)\big) ^{K\times K}.
\]
 Let $\Psi^{-\infty}(S)$ be the space of smoothing operators on $X$. Then we consider the subspace $\mathcal{A}^c_G(X)\subset \mathcal{A}_G(X)$ with
\[
\mathcal{A}^c_G(X):=\big( C^\infty_c(G) \hat{\otimes} \Psi^{-\infty} (S)\big) ^{K\times K}\subset \mathcal{A}_G(X). 
\]

This is an algebra and it  corresponds, under Abels' identification, to the algebra of smoothing  $G$-equivariant operators on $X$ 
of $G$-compact support. Notice that
the latter or, equivalently,  $\mathcal{A}^c_G(X)$ is a subalgebra of the Roe algebra $C^*_G (X)$ in a natural way.
More generally, if $E$ is a $G$-equivariant vector bundle on $X$, define 
\[
\mathcal{A}^c_G(X, E):= \big( C^\infty_c(G) \hat{\otimes} \Psi^{-\infty}(S, E|_S)\big) ^{K\times K},
\]
where $\Psi^{-\infty}(S, E|_S)$ is the space of smoothing operators on $E|_S$.  We remark that $C^\infty_c(G)$ and $\Psi^{-\infty}(S)$ are Fr\'echet algebras, and the tensor products in the above definitions of $\mathcal{A}^c_G(X)$  and $\mathcal{A}^c_G(X, E)$ are projective tensor products. The following 
 description of  $\mathcal{A}^c_G(X, E)$ is proved in \cite[Prop. 1.7]{PP2}
\begin{equation}\label{iso-c}
\begin{split}
&\mathcal{A}^c_G(X, E)\\
\cong &\big\{ \Phi: G\to \Psi^{-\infty}(S, E|_S),\text{\ smooth, compactly supported and }K\times K \text{\ invariant}\big\}.
\end{split}
\end{equation}
\subsection{The Harish-Chandra Schwartz algebra $\mathcal{C}(G)$}
\label{subsec:Harish-Chandra}
Let $\widehat{G}$ be the tempered dual of isomorphism classes of unitary irreducible representations of $C^*_r(G)$. In noncommutative geometry, $C^*_r(G)$ can be viewed as the algebra of ``continuous functions" on $\widehat{G}$. In the following, we introduce the noncommutative version of smooth functions on $\widehat{G}$ 
 which is the Harish-Chandra Schwartz algebra $\mathcal{C}(G)$. 

Let $v_0$ be a unit vector in the spherical representation  $\pi$ of $G$, and $\Theta$ be the Harish-Chandra spherical function on $G$ defined by 
\[
\Theta(g):=\langle v_0,  \pi(g)v_0\rangle. 
\]

Let $\mathfrak{g}$ be the Lie algebra of $G$ and $U(\mathfrak{g})$ be the universal enveloping algebra associated to $\mathfrak{g}$.  For $V, W\in U(\mathfrak{g})$, let $L_V$ be  the differential operator  on $G$ associated to the left $G$-action on $G$, and $R_W$ be the differential operator on $G$ associated to the right $G$-action on $G$.  

Consider the Cartan decomposition $G=K\exp\mathfrak{p}$ where $\mathfrak{p}$ is the Lie algebra of the associated parabolic subgroup $P$. For $X\in \mathfrak{p}$, let $||X||$ be a $K$-invariant Euclidean norm on $\mathfrak{p}$. Let $L: G\to \mathbb{R}^+$ be the function on $G$ defined by,
\[
L(g):=||X|| 
\]
for $g=k\exp(X)$ with $k\in K$ and $X\in \mathfrak{p}$.  

\begin{definition}\label{defn:HC}  Define the seminorm $\nu_{V, W, m}(-)$ by
\[
\nu_{V, W, m}(f):= \sup_{g\in G} \lVert |1+L(g)|^m \Theta(g)^{-1} L_V R_W\big(f\big)(g) \rVert,\ f\in C_c^\infty(G). 
\] 
The Harish-Chandra Schwartz algebra\footnote{In the latest version of our paper \cite{PPST}, inspired by \cite{Lafforgue},  we found it more convenient to work with a Banach algebra version of the Harish-Chandra Schwartz algebra, which we call the Lafforgue algebra. We refer the reader to \cite{PPST} for its precise definition.}  $\mathcal{C}(G)$ for $G$ is the space of smooth functions $f$ on $G$ such that 
\[
\nu_{V,W, m}(f)<\infty, \;\;\text{for} \;\;V, W\in U(\mathfrak{g}), m\geq 0. 
\]
\end{definition}

By \cite{Lafforgue}, the Harish-Chandra Schwartz algebra $\mathcal{C}(G)$ is a subalgebra of $C^*_r(G)$ stable under holomorphic functional calculus. Therefore, we have
\[
K_i(\mathcal{C}(G))\cong K_i(C^*_r(G)), \ i=0,1. 
\]

Starting from the Harish-Chandra Schwartz algebra $\mathcal{C}(G)$ we can also define an algebra of smoothing $G$-equivariant operators on $X$
by considering 
\[
\mathcal{A}^\infty_G(X, E):= \big( \mathcal{C}(G) \hat{\otimes} \Psi^{-\infty}(S, E|_S)\big) ^{K\times K}.
\]
One can extend \eqref{iso-c} and prove the following description of $\mathcal{A}^\infty_G(X, E)$. 
\begin{equation}\label{descrption-smoothings-HC}
\mathcal{A}^\infty_G(X,E)\cong\left\{\begin{array}{l}\Phi:G\to \Psi^{-\infty}(S),~K\times K~\mbox{invariant and}\\
~g\mapsto 
\nu_{V,W,m} (\| \Phi(g)\|_\alpha)\;\;\text{bounded}\;\;~\forall \alpha,m,V, W\end{array} \right\}
\end{equation}
The algebra $\mathcal{A}^\infty_G(X,E)$ is a subalgebra of the Roe $C^*$-algebra $C^*_G (X,E)$ and it is not difficult to prove, see \cite[Proposition 3.9]{PP1}, that it is dense and holomorphically closed.
Thus we have
\[
K_i(\mathcal{A}^\infty_G(X, E))\cong K_i(C^*_G (X,E)), \ i=0,1. 
\]

We end this subsection by pointing out that there exists a Morita isomorphism 
 $\mathcal{M}$ between $C^*_G (M,E)$ and  $C^*_r (G)$; this can be 
 justified by general principles, since
 $C^*_G (M,E)$ is isomorphic to the $C^*$-algebra of compact operators $\mathbb{K} (\mathcal{E})$, with
   $\mathcal{E}$ denoting  the $C^*_r (G)$-Hilbert module obtained by closing the space of compactly supported 
sections of $E$ on $X$, $C^\infty_c (X,E)$, endowed with the  $C^*_r G$-valued inner product
 $(e,e^\prime)_{C^*_r G} (g):= (e,g\cdot e^\prime)_{L^2 (X,E)}, \quad e,e^\prime \in C^\infty_c (X,E)\,,\;g\in G$ (and it is well known that
 then $K_i(\mathbb{K} (\mathcal{E}))=K_* (C^*_r (G))$.)
Following Hochs-Wang \cite{Hochs-Wang-HC}, we prefer to implement explicitly this isomorphism as follows: we consider 
$\mathcal{A}^\infty_G (M,E)$ and  $\mathcal{C}(G)$ and
consider  a partial trace map $\Tr_S: \mathcal{A}^\infty_G (X,E)\to \mathcal{C}(G)$
associated to the slice $S$: if $f\otimes k\in (\mathcal{C}(G))\hat{\otimes}\Psi^{-\infty}(S,E|_S))^{K\times K}$ then 
\[
\Tr_S (f\otimes k):= f \Tr (T_k)=f\int_S \tr k(s,s)ds,
\] 
with $T_k$ denoting the smoothing operator on $S$ defined by $k$ and
$\Tr (T_k)$ its functional analytic trace on $L^2 (S,E|_S)$. It is proved in  \cite{Hochs-Wang-HC} that this map
induces the isomorphism $\mathcal{M}$ between $K_* (C^*_G (X,E))$ and  $K_* (C^*_r (G))$.

\subsection{The index class of a $G$-invariant elliptic operator}
Recall  that  $X$ is  even-dimensional. Let $D$ 
be an odd $\ZZ_2$-graded Dirac operator, equivariant with respect to the $G$-action.
Recall, first of all,  the classical Connes-Skandalis idempotent. Let $Q$ be a $G$-equivariant parametrix
of $G$-compact support with remainders $S_\pm\in \mathcal{A}^c (X,E)$; consider the $2\times 2$ idempotent
\begin{equation}\label{CS-projector}
P_{Q}:= \left(\begin{array}{cc} S_{+}^2 & S_{+}  (I+S_{+}) Q\\ S_{-} D^+ &
I-S_{-}^2 \end{array} \right).
\end{equation}
This produces a well-defined class 
\begin{equation}\label{CS-class}
\operatorname{Ind_c} (D):= [P_{Q}] - [e_1]\in K_0 (\mathcal{A}_G^c (X,E))\;\;\text{with}\;\;e_1:=\left( \begin{array}{cc} 0 & 0 \\ 0&1
\end{array} \right).
\end{equation}

\begin{definition}
The $C^*$-index associated to $D$ is the class $\Ind_{C^*_G(M,E)} (D)\in K_0 (C^*_G (X,E))$ obtained by 
considering $ [P_{Q}] - [e_1]$ as a formal difference of idempotents with entries in $C^*_G (X,E)$,
under the continuous inclusion $\iota : \mathcal{A}_G^{c} (X,E)\hookrightarrow C^*_G (X,E)$.\\
\end{definition}
One can also give a definition of $\Ind_{C^*_G(X,E)} (D)\in K_0 (C^* (X,E)^G)$ using Coarse Index Theory,  as in the book of Higson and Roe, see
\cite{hr-book};
the compatibility of the two definitions is proved in \cite[Proposition 2.1]{PS-Stolz}.

We denote the image through the Morita isomorphism $\mathcal{M}$ of the index class $\Ind_{C^*_G(X,E)}  (D)\in K_0 (C^*_G (X,E))$ in the group
$K_0 (C^*_r (G))$ by $\Ind_{C^*_r (G)} (D)$. There are other, well-known descriptions of the latter index class:
one, following Kasparov, see  \cite{kasparov-functor}, describes the $C^*_r (G)$-index class
as the difference of two finitely generated projective $C^*_r (G)$-modules, using the invertibility modulo $C^*_r (G)$-compact
operators of (the bounded-transform of) $D$; the other description
is via assembly  and $KK$-theory, as in the classic article by Baum, Connes and Higson \cite{BCH}.
All these descriptions of the class $\Ind_{C^*_r (G)} (D)\in K_0 (C^*_r (G)) $ are equivalent. See \cite{roe-comparing}
and \cite[Proposition 2.1]{PS-Stolz}.

\medskip
\noindent
There is another way of expressing the index class $\Ind_{C^*_G(X,E)}  (D)\in K_0 (C^*_G (X,E))$; this is  due to Connes--Moscovici \cite{ConnesMoscovici}
and employs  the parametrix
\begin{equation}
 Q
:= \frac{I-\exp(-\frac{1}{2} D^- D^+)}{D^- D^+} D^+
\end{equation}
with
$I-Q D^+ = \exp(-\frac{1}{2} D^- D^+)$, $I-D^+ Q =  \exp(-\frac{1}{2} D^+ D^-)$. This particular choice of parametrix produces
 the idempotent
\begin{equation}
\label{cm-idempotent1}
V_{{\rm CM}} (D)=\left( \begin{array}{cc} e^{-D^- D^+} & e^{-\frac{1}{2}D^- D^+}
\left( \frac{I- e^{-D^- D^+}}{D^- D^+} \right) D^-\\
e^{-\frac{1}{2}D^+ D^-}D^+& I- e^{-D^+ D^-}
\end{array} \right)
\end{equation}
where ${\rm CM}$ stands for Connes and Moscovici. We certainly have $\Ind_{C^*_G(M,E)} (D)
= [V_{{\rm CM}} (D)] - [e_1]$.\\

The following result is proved in Piazza-Posthuma\footnote{In \cite{Hochs-Wang-HC} and \cite{hst} the authors established, through a 
different argument with respect to the one provided in \cite{PP2}, that $V_{{\rm CM}} (D_{\rm split})$
is an element in the smooth algebra $M_{2\times 2} (\mathcal{A}_G^\infty (X,E))$
(with the identity adjoined). However, $D$ and $D_{\rm split}$ are not identical. One might complete this proof of Proposition \ref{prop:CM-rapid} by applying an argument involving the Volterra series connecting $V_{{\rm CM}} (D)$ to $V_{{\rm CM}} (D_{\rm split})$. } \cite{PP2}: 
\begin{proposition}\label{prop:CM-rapid}
The idempotent $V_{{\rm CM}} (D)$  is an element in  $M_{2\times 2} (\mathcal{A}_G^\infty (X,E))$
(with the identity adjoined).
\end{proposition}

\noindent
As $\mathcal{A}^\infty (M,E)$ is holomorphically closed
 in $C^*_G(X,E)$ we have
  \begin{equation}\label{CM=index-bis}
  \begin{split}
&\Ind_{C^*_G(X,E)} (D)\equiv  \Ind_{\mathcal{A}^\infty (M,E)} (D)\\
=& [V_{{\rm CM}} (D)]-[e_1]\in K_0 (\mathcal{A}^\infty (M,E))=
K_0 (C^* (M,E)^G).
\end{split}
\end{equation}
We call $ [V_{{\rm CM}} (D)]-[e_1]\in K_0 (\mathcal{A}^\infty (M,E))$ the {\it smooth} index class associated to $D$. 

\medskip
As in \cite{moscovici-wu}, \cite{GMPi} we shall also consider the adjoint $V^*_{{\rm CM}} (D)$, which is again an element with entries  in  $\mathcal{A}_G^\infty (X,E)$. It is 
easy to prove that $ [V_{{\rm CM}} (D)]-[e_1]=  [V^*_{{\rm CM}} (D)]-[e_1] \quad\text{in}\quad  K_0 (\mathcal{A}^\infty_G (X,E))=K_0 (C^*_G (X,E))$
through an explicit  homotopy.
\begin{definition}\label{def:true-CM}
We set $V(D): = V_{{\rm CM}} (D)\oplus V^*_{{\rm CM}} (D).$
\end{definition}
Let $f_1=e_1\oplus e_1$.
Then  
\begin{equation}\label{twice}
[V(D)]-[f_1] = 2  ( [V_{{\rm CM}} (D)]-[e_1])\equiv 2\Ind_{C^*(M,E)^G} (D)\quad\text{in}\quad K_0 (C^*_G (X,E)).
\end{equation}
We call $V(D)$ the symmetrized Connes-Moscovici projector;  as we shall see, it plays an important role in the proof of explicit higher index formulae.

Let $X=G/K$ with the proper left $G$ action.  For a highest weight $\mu$ of $K$, let $V_\mu$ be the corresponding unitary irreducible $K$-representation associated to $\mu$, and $\widetilde{V}_\mu$ be the associated vector bundle on $G/K$ defined by $E_\mu:=G\times V_\mu/K$. We consider the Dirac operator $D_\mu$ associated to the vector bundle $E_\mu$ on $X$. Applying the construction in Definition \ref{def:true-CM}, we obtain an element $\Ind_{C^*_G(X, E_\mu)}(D_\mu)\in K_0(C^*_G(X, E_\mu)^G)=K_0(C^*_r(G))$. Varying $\mu$, we obtain the following morphism of abelian groups
\[
\Ind: \mathfrak{Rep}(K)\to K_0(C^*_r(G)),
\]
where $\mathfrak{Rep}(K)$ is the representation ring of the compact group $K$.

Connes and Kasparov conjectured in the early 80s that the above index map $\Ind$ is an isomorphism, \cite{BCH}, providing a geometric approach to compute the $K$-theory groups of $C^*_r(G)$. 
 The study of the index map has been one of the driving forces in the study of higher index theory on proper cocompact $G$-manifolds. The isomorphism theorem is now established in great generality: it is satisfied by all almost connected topological groups \cite{nest-et-al}.

\section{Cyclic cocycles on $\mathcal{C}(G)$} 
\label{sec:cyclic cocycles}

In this section, we present two methods to construct cyclic cocycles on $\mathcal{C}(G)$. 
\subsection{Cyclic cohomology}
Cyclic cohomology was introduced by Alain Connes  \cite{Connes:IHES} as the noncommutative version of de Rham cohomology. We briefly recall it below. 
\begin{definition}
Let $A$ be a Fr\'echet algebra over $\mathbb{C}$.  The space of Hochschild cochains of degree $k$ of $A$ is defined to be the space
\[
C^k(A) \colon = \mathrm{Hom}_{\mathbb{C}}\big(A^{\otimes(k+1)}, \mathbb{C}\big)
\]
of all bounded $(k+1)$-linear functionals on $A$. The Hochschild codifferential $b \colon C^k(A) \to C^{k+1}(A)$ is defined by 
\[
\begin{aligned}
&b\Phi(a_0 \otimes \dots \otimes a_{k+1})  \\
=&\sum_{i=0}^k (-1)^i \Phi(a_0 \otimes \dots \otimes a_i a_{i+1} \otimes \dots \otimes a_{k+1}) + (-1)^{k+1} \Phi(a_{k+1}a_0 \otimes a_1 \otimes \dots \otimes a_{k}). 
\end{aligned}
\]
The Hochschild cohomology of $A$ is the cohomology of the complex $(C^\bullet(A), b)$. 
\end{definition}

\begin{definition}\label{defn:cyclic-coh}
A Hochschild $k$-cochain $\Phi \in C^k(A)$ is called \emph{cyclic} if 
\[
\Phi(a_k, a_0, \dots, a_{k-1}) = (-1)^k \Phi(a_0, a_1, \dots, a_k),\ \ \  \forall a_0, \dots, a_k \in A.
\]
The subspace $C^k_\lambda(A)$ of cyclic  cochains  is closed under the Hochschild codifferential. The cyclic cohomology $HC^k_\lambda(A)$ is  defined to be the cohomology of the subcomplex of  cyclic cochains. 

There exists a natural operator $S:HC^{k}_\lambda(A)\to HC^{k+2}_\lambda(A)$, the periodicity operator, and the periodic cyclic cohomology is defined as follows
\[
HP^{k}(A):=\lim_{n\to \infty} HC_\lambda^{k+2n}(A),\ k=0,1. 
\]
\end{definition}

To understand the definition of Hochschild and cyclic cohomology, we look at the definition for $k=0$. A degree 0 Hochschild cochain  $\Phi$ on $A$ is a linear functional on $A$. The coboundary of $\Phi$ is a $1$-cochain defined as follows
\[
b\Phi(a_0, a_1)=\Phi(a_0a_1)-\Phi(a_1a_0).
\]
The above formula for $b\Phi$ suggests that $\Phi$ is a Hochschild cocycle if $\Phi(a_0a_1)=\Phi(a_1a_0)$, $\forall a_0, a_1\in A$, i.e. $\Phi$ is a trace on $A$. For $k=0$, the cyclic property trivially holds. Hence, the degree 0 cyclic cohomology of $A$ consists of traces on $A$. In general, cyclic cohomology is a natural generalization of traces.

The cyclic cohomology can also be computed by the following $b$-$B$ bicomplex on $C^\bullet (A)$.
Consider the operator
$B : C^{k}(A)\rightarrow C^{k-1}(A)$ by the formula
\[
\begin{split}
  B \Phi( a_0\otimes \ldots \otimes a_{k-1} ) :=&
  \sum_{i=0}^{k-1}(-1)^{i(k-1)}\Phi(1, a_i, \cdots, a_{k-1}, a_0, \ldots, a_{i-1})\\
  &\qquad \qquad -(-1)^{i(k-1)}\Phi(a_i, 1, a_{i+1},\cdots, a_{k-1}, a_0,\cdots, a_{i-1}).
\end{split}
\]
This defines a differential, i.e., $B^2=0$, and we have $[B,b]=0$,
so we can form the $(b,B)$-bicomplex.
\[
\xymatrix{
\ldots&\ldots&\ldots\\
  C^2(A)\ar[u]^{b} \ar[r]^B & 
  C^1(A)\ar[u]^{b} \ar[r]^{B} &
  C^0(A) \ar[u]^b \\
  C^1(A) \ar[u]^{b} \ar[r]^B& C^0(A)\ar[u]^b\\
  C^0(A)\ar[u]^b
  }
\]
The cyclic cohomology $HC_\lambda^k(A)$ is isomorphic to the degree $k$ total cohomology of the above $b$-$B$ bicomplex, c.f. \cite[Sec. 2.4]{Loday}. 

For the application to index theory, we are interested in the pairing between cyclic cohomology and $K$-theory of $A$. Let $P = (p_{ij}), i, j = 1, \dots, m$ be an idempotent in $M_m(A)$, the space of $m\times m$ matrices with entries in $A$. 
The following formula  
\[
\langle [\Phi], [P] \rangle \colon = \frac{1}{k!}\sum_{i_0,\cdots, i_{2k}=1}^m\Phi(p_{i_0i_1}, p_{i_1i_2}, ..., p_{i_{2k}i_0})
\]
defines a natural pairing between $[\Phi]\in HP^{\text{even}}(A)$ and $K_0(A)$, i.e. 
\[
\langle \ \cdot \ , \  \cdot  \ \rangle \colon   HP ^{\text{even}}(A) \otimes K_0(A)  \to \mathbb{C}. 
\]
In this article, we are interested in constructing cyclic cocycles of the Harish-Chandra Schwartz algebra $\mathcal{C}(G)$ and applying them 
on the one hand to study the structure of $K_\bullet(\mathcal{C}(G))\cong K_\bullet(C^*_r(G))$ via the above pairing and, on the other hand, to obtain
and study
numeric invariants of a Dirac operator  through the pairing of  the associated $K$-theory index class with such cyclic cocycles.

Our study  is inspired by the computation of the periodic cyclic cohomology of the group algebra of a discrete group, \cite{Burghelea, Connes}. Let $\Gamma$ be a discrete group and $\mathbb{C}\Gamma$ be the group algebra, i.e. $\mathbb{C}\Gamma:=\bigoplus_{\gamma\in \Gamma}\mathbb{C}\gamma$ such that $\gamma_1\ast \gamma_2:=\gamma_1\gamma_2$.  \begin{equation}\label{eq:cyclic-discrete group}
HP^\bullet(\mathbb{C}\Gamma)= \prod_{\hat{\gamma}\in \langle \Gamma \rangle'}\big( H^\bullet (N_\gamma, \mathbb{C}) \otimes HP^\bullet (\mathbb{C})\big) \times \prod_{\hat{\gamma} \in \langle \Gamma \rangle''} H^\bullet(N_\gamma, \mathbb{C}),
\end{equation}
where $\langle \Gamma\rangle$ is the set of conjugacy classes of $\Gamma$, $\langle\Gamma\rangle'\subset \langle \Gamma\rangle$ is the subset of classes of elements $\gamma\in \Gamma$ of finite order, $\langle \Gamma\rangle''$ its complement, and $N_\gamma$ is the normalizer of $\gamma$ defined by $C_\gamma/(\gamma)$, the quotient of the centralizer $C_\gamma$ of $\gamma$ by the subgroup $(\gamma)$ generated by $\gamma$.  Let $O_\gamma$ be the conjugacy class associated to $\gamma$.  On $\mathbb{C}\Gamma$, we have the following trace $\operatorname{tr}_\gamma$  associated to $O_\gamma$:
\begin{equation}\label{eq:trace}
\operatorname{tr}_\gamma(f)=\sum_{\alpha\in O_\gamma} f(\alpha)=\sum_{\beta\in G/C_\gamma} f(\beta \gamma \beta^{-1}),
\end{equation} 
where we recall that $C_\gamma$ is the centralizer of $\gamma$.  The trace $\operatorname{tr}_\gamma$ is a  degree 0 cyclic cocycle on $\mathbb{C}\Gamma$ associated to the conjugacy class $O_\gamma\in \langle \Gamma\rangle$. The computation (\ref{eq:cyclic-discrete group}) shows that there are many interesting higher degree cyclic cocycles on  the group algebra $\mathbb{C}\Gamma$.

A complete generalization of the computation (\ref{eq:cyclic-discrete group}) for a general Lie group is still under search. In (\ref{eq:cyclic-discrete group}), the cyclic cohomology of $\mathbb{C}\Gamma$ is decomposed into conjugacy classes of the group $\Gamma$. As a step toward understanding the cyclic cohomology of $\mathcal{C}(G)$, we will discuss below recent developments in the following special cases. 

\begin{enumerate}
\item We shall study 
cyclic cocycles associated to the identity conjugacy class of $G$. In \cite{Pflaum-Posthuma-Tang:vanEst, PP1} a chain map was introduced from differentiable group cohomology of $G$ to the cyclic cohomology of $\mathcal{C}(G)$ generalizing the trace $\operatorname{tr}_{e}$ associated to the identity element $e$. 
\item We shall study the orbital integral associated to a conjugacy class of $G$, a direct generalization of \eqref{eq:trace}.
 Next, following \cite{st}, we shall define {\it higher} orbital integrals on $\mathcal{C}(G)$ and study their properties.
\end{enumerate}
\subsection{Case of $\mathbb{R}^n$}  \label{subsection:Rn}In this subsection, we look at the abelian group $\mathbb{R}^n$. The Harish-Chandra Schwartz algebra $\mathcal{C}(\mathbb{R}^n)$ consists of Schwartz functions on $\mathbb{R}^n$ with the convolution product
\[
f\ast g (x)=\int_{\mathbb{R}^n}f(y)g(x-y)dy. 
\]
Under the Fourier transform, 
\[
\mathcal{F}(f)(\xi)=\int_{\mathbb{R}^n} f(x)\exp^{-2\pi i \xi\cdot x} dx, 
\]
the image of $\mathcal{F}(f)$ is a Schwartz function on $\widehat{\mathbb{R}}^n$ for $f\in \mathcal{C}(\mathbb{R}^n)$, and the convolution product is transformed to the pointwise multiplication
\[
f\cdot g (\xi)=f(\xi)g(\xi),\ \mathcal{F}(f\ast g)=\mathcal{F}(f)\cdot \mathcal{F}(g). 
\]

Using the Fourier transform, we conclude that the cyclic cohomology of $\mathcal{C}(\mathbb{R}^n)$ is isomorphic to the one of $\mathcal{S}(\widehat{\mathbb{R}}^n)$, which by \cite[Theorem 46]{Connes:IHES} is computed as follows
\[
HP^i(\mathcal{C}(\mathbb{R}^n))=\left\{\begin{array}{ll}
								0,&i\not\equiv n\ (mod\ 2),\\
								\mathbb{C},&i\equiv n\ (mod\ 2).
								\end{array}\right.
\]

On $\mathcal{S}(\widehat{\mathbb{R}}^n)$, the $n$-th cyclic cohomology group is generated by the following cocycle
\[
\Psi(\hat{f}_0, \cdots, \hat{f}_n)=\int_{\mathbb{R}^n} \hat{f}_0d\hat{f}_1\cdots d\hat{f}_n. 
\]

On $\mathcal{C}(\mathbb{R}^n)$, the Fourier transform of $\Psi$ has the following form. Define a function $C:\underbrace{\mathbb{R}^n\times \cdots \times \mathbb{R}^n}_n \to \mathbb{R}$ by 
\begin{equation}\label{eq:defn-c}
C(x_1, \cdots, x_n):= \left |\begin{array}{lll}x_1^1&\cdots &x_1^n\\ &\cdots & \\ x_n^1&\cdots &x_n^n\end{array}\right |,
\end{equation}
where we have written $x\in \mathbb{R}^n$ as $x=(x^1, \cdots, x^n)$. 

Define $\Phi$ to be a cocycle on $\mathcal{C}(\mathbb{R}^n)$ by
\begin{equation*}
\begin{split}
\Phi(f_0, \cdots, f_n):=&\int_{\mathbb{R}^n}\cdots \int_{\mathbb{R}^n}d x_1\cdots dx_n \\
&\qquad\qquad C(x_1, \cdots, x_n)f_0(-x_1-\cdots-x_n)f_1(x_1)\cdots f_n(x_n).
\end{split}
\end{equation*}
It is by direct computation that one can show that up to a constant $\Phi$ is the Fourier transform of the cyclic cocycle $\Psi$ on $\mathcal{S}(\widehat{\mathbb{R}}^n)$. 
\begin{proposition}\label{lem:rncocycle}
The cyclic cocycle $\Phi$ is the generator of the cyclic cohomology of $\mathcal{C}(\mathbb{R}^n)$. 
\end{proposition}
The generalization of $\Phi$ to general groups plays the central role in the following constructions. 
\subsection{Differentiable group cohomology} We observe that the function $C$ defined in Eq. (\ref{eq:defn-c}) is a group cocycle on $\mathbb{R}^n$. And the correspondence, $C\mapsto \Phi$, can be viewed as the analog of $H^\bullet(N_\gamma)$ in the cyclic cohomology of $\mathcal{C}(G)$ for $\gamma=e$ and $G=\mathbb{R}^n$ with $N_{e}=G$ . In \cite{Pflaum-Posthuma-Tang:vanEst, PP1}, this construction was generalized to general Lie group(oid)s.  In literature, two different but equivalent models have been used in the definition of differentiable group cohomology. We recall them below.

\begin{definition}\label{defn:group-coh-normalized}
Let $C^\infty(G^{\times k})$ be the space of smooth functions on $\underbrace{G\times \cdots \times G}_{k}$. Define the differential $\delta: C^\infty(G^{\times k})\to C^\infty(G^{\times k})$ by
\begin{eqnarray*}
&&\delta(\varphi)(g_1, \cdots, g_{k+1})\\
&=&\varphi(g_2, \cdots, g_{k+1})-\varphi(g_1g_2, \cdots, g_{k+1})\\
&+&\cdots+(-1)^k\varphi(g_1, \cdots, g_kg_{k+1})
+(-1)^{k+1}\varphi(g_1, \cdots, g_k).
\end{eqnarray*}
The differentiable group cohomology $H^\bullet_{\rm diff}(G)$ is defined to be the cohomology of $(C^\infty(G^{\times \bullet}), \delta)$, which is called the normalized differentiable group cohomology complex. 
\end{definition}

\begin{definition}\label{defn:groupcoh-homogeneous} 
\[
C^k_{\rm diff}(G):=\left\{\begin{array}{l}c:G^{\times(k+1)}\to\mathbb{C}~{\rm smooth},\\ ~c(gg_0,\ldots,gg_k)=c(g_0,\ldots,g_k),\ \forall g,g_0,\ldots,g_k\in G\end{array}\right\}.
\]
Define the differential $\partial:C^k_{\rm diff}(G)\to C^{k+1}_{\rm diff}(G)$  as
\[
(\partial c)(g_0,\ldots,g_{k+1}):=\sum_{i=0}^{k+1}(-1)^ic(g_1,\ldots,\hat{g}_i,\ldots,g_{k+1}).
\]
The differentiable group cohomology $H^\bullet_{\rm diff}(G)$ can also be computed by the cohomology of the chain complex $(C^\bullet_{\rm diff}(G), \partial)$, which is called the homogenous differentiable group cohomology complex. 
\end{definition}

In this paper, following the literature, we will use both chain complexes introduced in Definition \ref{defn:group-coh-normalized} and \ref{defn:groupcoh-homogeneous}. And we will explicitly point out the models used in the respective formulas below. 
 
 \smallskip
Fix a Haar measure $dg$ on $G$.  
\begin{definition}\label{defn:pairing}
Define a pairing between $C^\infty(G^{\times k})$ and $C^\infty_c(G)^{\hat{\otimes}(k+1)}$ by 
\begin{equation*}
\langle\hat{\varphi}, f_0\otimes\cdots\otimes f_k\rangle:=\int f_0(g_k^{-1}\cdots g_1^{-1})f_1(g_1)\cdots f_k(g_k)\varphi(g_1, \cdots, g_k)dg_1\cdots dg_k.
\end{equation*}
In the above formula, we have used the normalized differentiable cohomology complex in Definition \ref{defn:group-coh-normalized}.
\end{definition}

\begin{theorem} \label{thm:character} Assume that $G$ is unimodular, i.e. the Haar measure is bi-invariant. The above pairing descends to a character morphism $\chi$,
\[
\chi: H^\bullet_{\rm diff}(G)\to HP^\bullet(\mathcal{C}(G)).
\]
\end{theorem}

\begin{remark} In \cite{Pflaum-Posthuma-Tang:vanEst}, the character morphism was introduced as a map from $H^\bullet_{\rm diff}(G)$ to $HP^\bullet(C^\infty_c(G))$, where $C^\infty_c(G)$ is the convolution algebra of compactly supported functions. When $G$ has property $RD$ and the homogeneous space $G/K$ has nonpositive sectional curvature, Piazza and Posthuma   \cite{PP1}  improved this character morphism as a map from $H^\bullet_{\rm diff}(G)$ to the cyclic cohomology of $H^\infty_L(G)$, which for a length function $L$ on $G$ is defined as follows,  
\[
H^\infty_L(G)=\left\{f\in L^2(G)|\int_G \big(1+L(g)\big)^{2k} |f(g)|^2dg <\infty,\ \forall k\in \mathbb{N}\right\}. 
\]
For a connected reductive linear Lie group $G$,  $G/K$ is a Hermitian symmetric space of noncompact type \cite{Helg}, i.e. $G/K$ has nonpositive sectional curvature. We observe that the same argument for $H^\infty_L(G)$ can be applied to the Harish-Chandra Schwarz algebra $\mathcal{C}(G)$ via the techniques in \cite[Appendix A.]{st}.  This is how we reached the final statement for Theorem \ref{thm:character}. 
\end{remark}

\begin{example} \label{ex:sl2} For $G=SL_2(\mathbb{R})$,  by the van Est isomorphism, $H^2_{\rm diff}(SL_2(\mathbb{R}))$ is generated by the area cocycle,  which has the following geometric description. Let $\mathbb{H}$ be the upper half space identified with the homogenous space $SL_2(\mathbb{R})/SO(2)$. Let $[e]$ be the point in $\mathbb{H}$ corresponding to the coset $e SO(2)$ in $SL_2(\mathbb{R})/SO(2)$. As $\mathbb{H}$ is equipped with a metric of  constant negative curvature, any two points in $\mathbb{H}$ are connected by a unique geodesic. Given $g_1, g_2, g_3\in SL_2(\mathbb{R})$, we consider $g_i[e]$, $i=1,2,3$, in $\mathbb{H}$, and the corresponding geodesic triangle  $\Delta(g_1[e], g_2[e], g_3[e])$ with vertices $g_1[e], g_2[e], g_3[e]$. Define a smooth function $A$ on $SL_2(\mathbb{R})\times SL_2(\mathbb{R})\times SL_2(\mathbb{R})$ as follows,
\[
A(g_0,g_1,g_2):={\rm Area}_{\mathbb{H}}(\Delta(g_0[e],g_1[e],g_2[e])),
\]
where $g_i\in SL(2,\RR)$ acts on $\mathbb{H}$ as usual by M\"obius transformations. It is straightforward to check that $A$ is a smooth 2-cocycle on $SL_2(\mathbb{R})$ (in the homogeneous differentiable group cohomology chain complex introduced in Definition \ref{defn:groupcoh-homogeneous}), c.f. \cite[Remark 2.20, (ii)]{PP1}, generalizing the volume cocycle  (\ref{eq:defn-c}) for $\mathbb{R}^n$ in Sec. \ref{subsection:Rn}.  Furthermore, its image $\chi(A)$ under the character morphism can be identified with the Connes-Chern character of the fundamental $K$-cycle of $SL_2(\mathbb{R})$ in \cite[I.9, Lemma 5]{Connes:IHES}. For $f_0, f_1, f_2\in \mathcal{C}(SL_2(\mathbb{R}))$, 
\[
\begin{split}
\chi(A)(f_0, f_1, f_2)=&\int_{SL_2(\mathbb{R})}\int_{SL_2(\mathbb{R})} f_0\big((g_1g_2)^{-1}\big) f_1(g_1) f_2(g_2)\\
 &\qquad \qquad \operatorname{Area}_{\mathbb{H}}\big( \Delta([e], g_1[e], g_1g_2[e])\big)dg_1 dg_2.
\end{split}
\] 
\end{example}

\subsection{Higher orbital integrals} The trace $\operatorname{tr}_\gamma$ introduced in Equation (\ref{eq:trace}) on $\mathbb{C}\Gamma$ has a natural generalization $\operatorname{tr}_g$ for $G$, i.e.
\begin{equation}\label{orbital-i}
\operatorname{tr}_g(f):=\int_{G/Z_g}f(hgh^{-1}) d(hZ_g). 
\end{equation}

For a discrete group $\Gamma$, it is not clear whether the trace $\operatorname{tr}_\gamma$  on $\mathbb{C}\Gamma$, for a general element $\gamma$, pairs with $K_0(C^*_r(\Gamma))$. Still, although a general statement is missing, the pairing is well-defined 
for groups of polynomial growths and, more importantly, for Gromov hyperbolic groups.
The latter example has been a challenge for some time and has been settled in the affirmative way by Puschnigg, who defined
a dense holomorphically closed subalgebra of  $C^*_r(\Gamma)$ to which the trace $\operatorname{tr}_\gamma$ extends.
See \cite{Puschnigg}. 
Higher cyclic cocycles associated to the other elements in the  Burghelea's decomposition were studied in detail
 by Chen-Wang-Xie-Yu in \cite{ChenWangXieYu}
and by Piazza-Schick-Zenobi in \cite{PSZ}; in the above articles (and references therein) many purely geometric applications were also given.
Notice that there is an analogy between the Puschnigg algebra and the extendability of delocalized (higher) cyclic 
cocycles for discrete Gromov hyperbolic groups and the Harish-Chandra  Schwartz algebra and the extendability of delocalized
(higher) cyclic cocycles for connected real reductive Lie groups.

\smallskip
We go back to the orbital integral \eqref{orbital-i}. Hochs and Wang \cite{Hochs-Wang-HC}, building on 
results of Harish-Chandra \cite{Harish-Chandra-dis},
proved that  for a semisimple element\footnote{that is, the corresponding operator $\operatorname{Ad}_g: \mathfrak{g}\to \mathfrak{g}$ is diagonalizable.}
$g$ 
 the trace functional $\operatorname{tr}_g$ is well defined on  the Harish-Chandra Schwartz algebra $\mathcal{C}(G)$ and therefore pairs with $K_0(\mathcal{C}(G))\cong K_0(C^*_r(G))$.  Hochs and Wang \cite{Hochs-Wang-HC} computed the pairing between $\operatorname{tr}_g$ and the index element $\Ind_{C^*_r(G)}(D)$ for a $G$-equivariant Dirac operator on a proper cocompact $G$-manifold $X$; this explicit formula will be recalled in Theorem \ref{thm:HochsWang-HC} below. As we shall see in Theorem \ref{thm:HochsWang-HC}, when $G$ does not have a compact Cartan subgroup the trace $\operatorname{tr}_g$ is a trivial cyclic cohomology class.
 This raises a natural question to find a generalization of $\operatorname{tr}_g$ that pairs nontrivially with $K_0(C^*_r(G))$. The problem
 was raised and solved in \cite{st}, as we shall now explain.

\smallskip
Let $K<G$ be a maximal compact subgroup and let $P<G$, $P=MAN$, be a cuspidal parabolic subgroup of $G$. Using the Iwasawa decomposition $G=KMAN$ we write an element $g\in G$ as 
\[
g = \kappa (g)\mu(g) e^{H(g)}n \in KM A N = G. 
\]
We observe that the function $H(g)$ is well defined on $G$ though the components $\kappa(g)$ and $\mu(g)$ in the Iwasawa decomposition may not be unique. Let $\operatorname{dim}(A)=m$.  Choosing coordinates of the Lie algebra $\mathfrak{a}$ of $A$,  we define the function 
\[
H=(H_1, \dots, H_m): G\to \mathfrak{a}. 
\] 
Song and Tang defined in \cite{st} the following cyclic cocycle on  $\mathcal{C}(G)$, the Harish-Chandra Schwartz algebra of $G$.

\begin{definition}\label{defn:higherorbitalintegral}
For $f_0, ..., f_m \in \mathcal{C}(G)$ and a semi-simple element $g\in M$,  define $\Phi^{P}_{g}$ by the following integral, 
\begin{equation}
\label{eq:PhiPg}
\begin{aligned}
\Phi^{P}_{g}(f_0, f_1, &\dots, f_m) \\
:=&\int_{h \in M/Z_M(x)} \int_{K N} \int_{G^{\times m}} C\big( H(g_1...g_mk), H(g_2...g_mk ),\cdots, H(g_mk)\big) \\
&\qquad f_0 \big (kh g h^{-1}nk^{-1} (g_1\dots g_m)^{-1}\big)f_1(g_1) \dots f_m(g_m)dg_1\cdots dg_m dk dn dh,  
\end{aligned}
\end{equation}
where $Z_M(x)$ is the centralizer of $x$ in $M$. 
\end{definition}

It is easy to check that the integral in Equation (\ref{eq:PhiPg}) is convergent for $f_0,..., f_m\in C^\infty_c(G)$. Song and Tang proved \cite[Theorem 3.5]{st} the following property for $\Phi^P_g$. 
\begin{theorem}\label{thm:higherorbital} For a cuspidal parabolic subgroup $P=MAN$ of $G$, and a semisimple element $g\in M$, the cochain $\Phi^P_g$ is a cyclic cocycle on $\mathcal{C}(G)$. 
\end{theorem}

\subsection{From cocycles on  $\mathcal{C}(G)$ to cocycles on $\mathcal{A}_G^\infty (X,E)$}

In Section \ref{subsec:Harish-Chandra}, we introduced $\mathcal{A}^{\infty}_G(X, E)$, a dense subalgebra of the Roe algebra $C^*_G(X, E)$ that has the same $K$-theory groups as $\mathcal{C}(G)$.  In this subsection, we explain how to lift the cocycles on $\mathcal{C}(G)$ to $\mathcal{A}_G^{\infty}(X, E)$. 

Recall, see \eqref{descrption-smoothings-HC}, that an element in $\mathcal{A}_G^\infty(X, E)$ is a $K$ bi-invariant smooth function on $G$ with values in $\Psi^{-\infty}(S)$.  The partial trace map 
\[
\Tr_S: \mathcal{A}^\infty_G(X, E)\to \mathcal{C}(G),\ 
\Tr_S (A)(g):=\Tr (A(g))=\int_S \tr\big(A(g)(s,s)\big) ds. 
\]
Generalizing this partial trace map, we define the map $$\tau: C^k(\mathcal{C}(G))\to C^k(\mathcal{A}_G^\infty(X, E))$$ as follows,
\[
\tau_\varphi (A_0, \cdots, A_k):=\varphi\big(\Tr\big(A_0(g_0)\circ A_1(g_1)\circ \cdots \circ A_k(g_k) \big)\big).
\]
It is easy to check that $\tau$ induces a chain map on the Hochschild and cyclic complexes. Thus, see \cite{PP1},
we have:
\begin{proposition}\label{prop:morphism}
The chain map $\tau$ induces a morphism $\tau: HP^\bullet (\mathcal{C}(G))\to HP^\bullet(\mathcal{A}_G^\infty(X, E))$. 
\end{proposition}

 We shall  work with the following cyclic cocycles on $\mathcal{A}^\infty_G(X, E)$:
 \begin{itemize}
 \item for a differentiable group cocycle $\varphi$ of $G$, we shall consider  $\tau\big(\chi(\varphi))$, denoted in the sequel by
 $\tau^X_\varphi$;
 \item for $g\in G$, we shall consider $\tau(\tr_g)$, denoted by  $\tr^X_g$ in the sequel;
 \item for $g\in G$, we shall consider  $\tau\big(\Phi^P_{g}\big)$, denoted by $\Phi^P_{X,g}$ in the sequel.
 \end{itemize} 
 
\section{Higher indices for proper cocompact manifolds without boundaries}
\label{sec:higher indices}

In this section, we present higher index theorems on proper cocompact $G$-manifolds associated to the cyclic cocycles on $\mathcal{C}(G)$ introduced in Section \ref{sec:cyclic cocycles}. 

\subsection{$L^2$-index theorem}

We start by reviewing  Atiyah's  $L^2$-index theorem on Galois coverings. Let $M$ be a compact smooth manifold without boundary, and $E$ a twisted spinor bundle on $M$, and $D$ 
an odd $\ZZ_2$-graded Dirac operator acting on the sections of $E$. 

Let $\Gamma$ be the fundamental group of $M$ and $X$ be the universal covering space of $M$. $X$ is equipped with a proper, free, and cocompact $\Gamma$ action such that the quotient is $M$.  As $X$ is a covering space of $M$, the operator $D$, as a differential operator on $M$, lifts to an odd $\ZZ_2$-graded $\Gamma$-equivariant operator $\widetilde{D}$ on $\widetilde{E}$, the pullback of $E$ to $X$.  

Let $C^*_r(\Gamma)$ be the reduced group $C^*$-algebra of $\Gamma$. Following Equation (\ref{CM=index-bis}), we consider the index $\Ind_{C^*_r(\Gamma)}(\widetilde{D})$ of the operator $\widetilde{D}$, which is an element of $K_0(C^*_r(\Gamma))$. 

The trace $\operatorname{tr}_e$ on $\mathbb{C}\Gamma$ naturally extends naturally to a trace on $C^*_r(\Gamma)$. The pairing between $\operatorname{tr}_e$ and the index element $\operatorname{Ind}_{C^*_r(\Gamma)}(D)$ was computed by Atiyah \cite{Atiyah:L2} in the following theorem.
\begin{theorem}\label{thm:atiyah} 
\[
\langle \operatorname{tr}_e, \operatorname{Ind}_{C^*_r(\Gamma)}(\widetilde{D})\rangle=
\operatorname{ind}(D). 
\]
where on the right hand side we have the Fredholm index of $D$.
\end{theorem}

Generalizing Atiyah's  Theorem \ref{thm:atiyah} to proper Lie group actions,  Connes and Moscovici \cite{Connes-Moscovici:L2} proved the following index 
formula on homogeneous spaces. Let $K$ be a maximal compact subgroup of $G$, and $X:=G/K$ be the associated homogeneous space, which is equipped with a proper left $G$ action with the quotient being a point. We assume that $X$ is equipped with a $G$-equivariant Spin$^c$-structure and consider the $G$-equivariant Dirac operator $D$ on $X$ obtained by twisting the Spin$^c$-Dirac operator with the bundle $E_\mu:=G\times V_\mu/K$, with $V_\mu$  an irreducible unitary $K$ representation associated with highest weight $\mu$. We obtain a well defined  index class for $D$, an  element $\Ind_{C^*_r(G)}(D)\in K_0(C^*_r(G))$. 

\begin{theorem} \label{thm:con-mos-L2}Assume that $G$ is unimodular. Let $\mathfrak{g}$ and $\mathfrak{k}$ be the Lie algebras of $G$ and $K$. 
\[
\langle \operatorname{tr}_e, \operatorname{Ind}(D_\mu)\rangle=\langle \widehat{A}(\mathfrak{g}, K)\wedge \operatorname{ch}(V_\mu)_{\mathfrak{p}^*}, [V]\rangle,
\]
where $\mathfrak{p}^*\subset \mathfrak{g}^*$ is the conormal space of $\mathfrak{k}$ in $\mathfrak{g}$, and $[V]$ is the fundamental class of $\mathfrak{p}^*$.
\end{theorem}

Hang Wang \cite{Wang:L2} generalized  the Connes-Moscovici theorem, Theorem \ref{thm:con-mos-L2}, to general proper cocompact $G$-manifolds as follows. 
\begin{theorem}\label{thm:HWang-L2} Let $X$ be a proper cocompact $G$-manifold which is equipped with a $G$-equivariant Spin$^c$-structure. Suppose that $D$ is a $G$-equivariant Dirac operator on $X$ associated to a twisted $G$-equivariant twisted spinor bundle $E=\mathcal{S}\otimes F$. The pairing between $\operatorname{tr}_e$ and $\Ind_{C^*_r(G)}(D)$ is computed as follows.
\[
\langle \operatorname{tr}_e, \Ind_{C^*_r(G)}(D) \rangle=\int_{X} c_{X} \operatorname{AS}(D),
\]
with $c_X$ a cut-off function for the proper action of $G$ on $X$ and 
\[
\operatorname{AS}(D):=\widehat{A}\Big(\frac{R_X}{2\pi i}\Big) \operatorname{tr}\Big( \exp\big(\frac{R_F}{2\pi i}\big)\Big).
\]
Here  $R_X$  is the curvature form for the Levi-Civita connection on $X$ and $R_F$ is the curvature form associated to a $G$-invariant Hermitian metric on $G$.
\end{theorem}

\subsection{Higher indices associated to differentiable group cohomology} In \cite{ConnesMoscovici}, Connes and Moscovici  established 
a higher version of  Atiyah's $L^2$-index Theorem \ref{thm:atiyah} and  computed a geometric formula  for the pairing between the group cohomology of the fundamental group $\Gamma$ and the index class $\Ind_{\mathcal{A}^c_{G}(X, E)}(D)\in K_\bullet(\mathcal{A}^c_G(X,E))$. Inspired by this result, Pflaum, Posthuma, and Tang \cite{Pflaum-Posthuma-Tang:vanEst, PPT}  proved 
a generalization of Theorem \ref{thm:HWang-L2}, computing the pairing between the cyclic cocycles in the image of the character morphism $\chi: H^\bullet_{\rm diff}(G)\to HP^\bullet(\mathcal{C}(G))$ in Theorem \ref{thm:character}
and the smooth index class $$\operatorname{Ind}_{\mathcal{C}(G)}(D)\in K_0 (\mathcal{C}(G))
 \equiv \operatorname{Ind}_{C^*_r(G)}(D)\in  K_0 (C^*(G)).$$ 
 
Let $X$ be a manifold equipped with proper $G$ action, and $\Omega^\bullet_{\text{inv}}(X)$ be the space of $G$-invariant differential forms on $X$. The de Rham differential restricts to $\Omega^\bullet_{\text{inv}}(X)$ and the associated cohomology is denoted by $H^\bullet_{\text{inv}}(X)$. 
\begin{definition}
The van Est morphism $\Phi: C^{k}_{\text{diff}}(G)\to \Omega^k_{\text{inv}}(X)$ is defined as follows. 
\[
\begin{split}
f_{\varphi}(x_0, \cdots, x_k)&:=\int_{G^{\times (k+1)}} c_X(g_0^{-1}x_0) \cdots c_X(g_k^{-1}x_k)\varphi(g_0, g_1, \cdots, g_k) {\rm d}g_0\cdots {\rm d}g_k,\\
\Phi(\varphi)&:=(d_{x_1}d_{x_2}\cdots d_{x_k} f_{\varphi})(x, \cdots, x).
\end{split}
\]
In the above formulas, we have used the homogeneous differentiable group cohomology chain complex in Definition \ref{defn:groupcoh-homogeneous}, and $f_\varphi$ is a smooth function on $X^{\times (k+1)}$, and $d_i f_{\varphi}$ is the differential of $f_{\varphi}$ along the $x_i$ component. As the cut-off function $c_X$ is compactly supported along each $G$ orbit, the above integral is convergent. 
\end{definition}
The following property is proved in \cite[Prop. 2.5]{PP1}.
\begin{proposition}\label{prop:vanest}
The map $\Phi$ induces a morphism $\Phi:H^\bullet_{\text{diff}}(G)\to H^\bullet_{\text{Inv}}(X)$.
\end{proposition}

The following theorem computes the explicit formula about the pairing between $\chi(\varphi)$ and $\Ind_{C^*_r(G)}(D)$.
\begin{theorem}\label{thm:PPT-index} 
Suppose that $G$ is unimodular. Let $G$ be a connected reductive linear Lie group acting properly and cocompactly on a manifold $X$. For any $[\varphi]\in 
H^{2k}_{\rm diff}(G)$, the index pairing is 
given by
 \[
\langle \chi(\varphi), \operatorname{Ind}_{C^*_r(G)}(D)\rangle
=\frac{1}{(2\pi\sqrt{-1})^k(2k)!}\int_{T^*X} c_X\,\Phi([\varphi])\wedge \widehat{A}(T^*X)\wedge\ch(F),
\]
where 
$c_X\in C^\infty_{\rm c}(X)$ is a cut-off function, $\widehat{A}(T^*X):=\widehat{A}\Big(\frac{R_X}{2\pi i}\Big)$ and  $\ch(F):=\operatorname{tr}\Big( \exp\big(\frac{R_F}{2\pi i}\big)\Big)$, as in Theorem \ref{thm:HWang-L2} 
\end{theorem}

\begin{remark} Pflaum, Posthuma, and Tang \cite{PPT} proved Theorem \ref{thm:PPT-index} through the algebraic index theorem method developed by Fedosov \cite{Fedosov:book} and \cite{Nest-Tsygan:algebraicindex} for general $\mathsf{G}$-invariant elliptic operators on $X$ on a proper cocompact $\mathsf{G}$-manifold for a Lie groupoid $\mathsf{G}$.  Recently, Piazza and Posthuma \cite{PP2} presented a new proof of this theorem for the case of Dirac operators using the heat kernel and Getzler's rescaling techniques.   
\end{remark}

\begin{remark} The assumption of unimodularity in Theorem \ref{thm:PPT-index}  can be dropped by working with smooth group cohomology of $G$ with coefficients. This is developed in \cite{Pflaum-Posthuma-Tang:vanEst}.  
\end{remark}

\begin{example}We consider the case in which $G$ is $\mathbb{R}^2$ and $K$ is trivial. In this case, $X=G/K$ is identified with $\mathbb{R}^2$. Via the isomorphism with $\mathbb{C}$, $\mathbb{R}^2$ is equipped with the $\mathbb{R}^2$ invariant Dolbeault operator $D$.  As there is no $L^2$-harmonic form on $\mathbb{R}^2$, it was observed by Connes and Moscovici \cite{Connes-Moscovici:L2} that the $L^2$-index $\langle \operatorname{tr}_e, \Ind_{C^*_r(\mathbb{R}^2)}(D)\rangle $ vanishes. However, the determinant $C$ function in Section \ref{subsection:Rn} is a 2-cocycle on $\mathbb{R}^2$, and its image $\chi(C)$ in $HP^{\text even}(\mathcal{C}(\mathbb{R}^2))$ coincides with the cyclic cocycle $\Phi$ introduced in Section \ref{subsection:Rn}.  We can apply Theorem \ref{thm:PPT-index} to compute 
\[
\langle\Phi,   \Ind_{C^*_r(\mathbb{R}^2)}(D)\rangle=\langle\chi(C), \Ind_{C^*_r(\mathbb{R}^2)}(D)\rangle=\frac{1}{2}. 
\]
This example shows that higher indices contain interesting information of the operator $D$ beyond the $L^2$-index. The computation also extends naturally to $\mathbb{R}^{2n}=\mathbb{C}^n$. 
\end{example}

\subsection{Delocalized indices}

Hochs and Wang \cite{Hochs-Wang-HC} computed the pairing between $\operatorname{tr}_g$ and $\Ind_{C^*_r(G)}(D)$ for a twisted Spin$^c$ Dirac operator on a proper cocompact $G$-manifold on $X$. 
To introduce their result, we fix the following set up. 
\begin{itemize}
\item $X^g$ is the $g$-fixed point submanifold; 
\item   $\mathcal{N}_{X^g}$ is the normal bundle of $X^g$ in $X$ and $R^\mathcal{N}$ is the curvature form associated to the Hermitian connection on $\mathcal{N}_{X^g}\otimes \mathbb{C}$;
\item  $L_{\operatorname{det}}$ is  the determinant line bundle of the Spin$^c$-structure on $X$ and $L_{\operatorname{det}}|_{X^g}$ is its restriction to $X^g$ and $R^{L}$ is the curvature form associated to the Hermitian connection on $L_{\operatorname{det}}|_{X^g}$;
\item $R_{X^g}$ is the Riemannian curvature form associated to the Levi-Civita connection on the tangent bundle of  $X^g$;
\item The ${\rm AS}_g(D)$ has the following expression for a twisted Dirac operator on $E=\mathcal{S}\otimes F$:
\begin{equation}\label{eq:X-geom-form}
{\rm AS}_g(D):=\frac{\widehat{A}\big(\frac{R_{X^g}}{2\pi i}\big) \operatorname{tr}\big(g \exp(\frac{R^F}{2\pi i})\big) \exp(\operatorname{tr}(\frac{R^L}{2\pi i}))}
{\operatorname{det}\big(1-g \exp(-\frac{R^{\mathcal{N}}}{2\pi i})\big)^{\frac{1}{2}}}.
\end{equation}
This will also be denoted by ${\rm AS}_g (X,E)$ or, if there is no confusion on the vector bundle $E$, simply by ${\rm AS}_g (X)$.
\end{itemize}

Hochs and Wang proved the following result using heat kernel and Getzler's rescaling techniques \footnote{We refer also
to the recent article \cite{PPST} for a related, detailed discussion
of this result.}

\begin{theorem}\label{thm:HochsWang-HC}
\begin{enumerate}
\item If $G$ has a compact Cartan subgroup, then 
\[
\langle \operatorname{tr}_g, \Ind_{C^*_r(G)}(D)\rangle= \int_{X^g} c_{X^g} \operatorname{AS}_g(D),
\]
for a cutoff function $c_{X^g}\in C^\infty_c(X^g)$  for the $Z_g$-action on $X^g$, with $Z_g$ equal to  the centralizer subgroup of $g$ in $G$;
\item If $G$ does not have a compact Cartan subgroup, 
\[
\langle \operatorname{tr}_g, \Ind_{C^*_r(G)}(D)\rangle= 0. 
\]
\end{enumerate}
\end{theorem}

\subsection{Delocalized higher indices}
To improve Theorem \ref{thm:HochsWang-HC} to allow $G$ to be nonequal rank, in \cite{hst}, Hochs, Song, and Tang computed the pairing between the cocycle $\Phi^P_g$ and the index element $\Ind_{C^*_r(G)}(D)$.  Let $X/AN$ be the quotient of $X$ with respect to the $AN<G$ action. The group $M$ acts properly and cocompactly on the quotient $X/AN$; for $g$ a semisimple element of $G$, $(X/AN)^g$ is the fixed submanifold of the $g$ action on the quotient $X/AN$; $c^g_{X/AN}$ is a smooth compactly supported cut-off function on $(X/AN)^g$. 

Let $\mathfrak{m}$ be the Lie algebra of the group $M$, and $\mathfrak{a}$ the Lie algebra of the group $A$. Using the $K$-invariant metric on $\mathfrak{p}$, we obtain a $K\cap M$-invariant decomposition
\[
\mathfrak{p}=(\mathfrak{p}\cap \mathfrak{m})\oplus \mathfrak{a}\oplus (\mathfrak{k}/(\mathfrak{k}\cap \mathfrak{m})),
\]
which induces the $K\cap M$ decomposition of spinors,
\[
S_\mathfrak{p}\cong S_{\mathfrak{p}\cap \mathfrak{m}}\otimes S_{\mathfrak{a}}\otimes S_{\mathfrak{k}(\mathfrak{k}\cap \mathfrak{m})}. 
\]
Using the slice theorem, Theorem. \ref{thm:slice}, $X/AN$ can be identified  as
$X/AN=M\times_{K\cap M} S$.  
On $X/AN$, we consider the bundle $E_{X/AN}$, 
\[
E_{X/AN}=M\times_{M\cap K}\big(S_{\mathfrak{p}\cap \mathfrak{m}}\otimes S_{\mathfrak{k}/(\mathfrak{k}\cap \mathfrak{m})}\otimes W|_S\big)\cong M\times _{M\cap K} \big( S_{\mathfrak{p}\cap \mathfrak{m}}\otimes \widetilde{W}\big),
\]
where we denote $S_{\mathfrak{k}/(\mathfrak{k}\cap \mathfrak{m})}\otimes W|_S$ by $\widetilde{W}$ and $M\times_{M\cap K}\widetilde{W}$ by $W_{X/AN}$.  We observe that $W_{X/AN}$ is an $M$-equivariant Hermitian vector bundle on $X/AN$. On $E_{X/AN}$, we consider the associated Dirac operator $D_{X/AN}$. The index of $D_{X/AN}$ is an element  $\Ind_{C^*_r(M)}(D_{X/AN})$ in $K_\bullet(\mathcal{C}(M))$.
 
Hochs, Song, and Tang \cite{hst} proved a reduction theorem relating the index pairing for $D$ and the index pairing for $D_{X/AN}$. 
\begin{theorem}\label{thm:reduction}
\[
\langle \Phi^P_g, \Ind_{C^*_r(G)}(D)\rangle=\langle \tr^M_g, \Ind_{C^*_r(M)}(D_{X/AN})\rangle. 
\]
\end{theorem}
  
The following theorem follows directly from Theorem \ref{thm:HochsWang-HC} and \ref{thm:reduction}.
\begin{theorem}\label{thm:HST} 
\[
\langle \Phi^P_g, \Ind_{C^*_r(G)} (D)\rangle =\int_{(X/AN)^g}c_{(X/AN)^g}{\rm AS}(X/AN)_g.
\]
\end{theorem}

\subsection{The example of $G/K$}

We look at the example of $X=G/K$ with the left $G$ action. Song and Tang \cite{st} prove the following theorem using Harish-Chandra's theory of orbital integrals;
it can also be derived from Theorem \ref{thm:HST} as a corollary. 
\begin{theorem}\label{thm:st-pairing}Let $H$ be a Cartan subgroup of $G$, and $T:=K\cap H$ with $t\in T$.  Assume that $T<M$ is a Cartan subgroup of $M$, and $P$ is a maximal cuspidal parabolic subgroup, i.e. $T$ is maximal among all possible $P$. Let $\Delta^M_T$ be the  corresponding Weyl denominator. 
\begin{enumerate}
\item 
\begin{equation}\label{thm:genrank}
\langle \Phi^P_{e}, \operatorname{Ind}_{C^*_r(G)}(D)\rangle =  \frac{1}{|W_{M \cap K}|} \cdot \sum_{w \in W_K} m\left(\sigma^{M}(w \cdot \mu)\right),
\end{equation}
where $\sigma^{M}(w \cdotp \mu)$ is the discrete series representation of $M$ with Harish-Chandra parameter $w \cdot \mu$, and $m\left(\sigma^{M}(w \cdot \mu)\right)$ is its Plancherel measure, and $W_{M\cap K}$ is the Weyl group of $M\cap K$;
\item 
\begin{equation}\label{eq:char}
\langle \Phi^P_{t}, \operatorname{Ind}(D) \rangle =  \frac{\sum_{w \in W_K} (-1)^w e^{w \cdot \mu}(t)}{\Delta^{M}_T(t)}.
\end{equation}
\end{enumerate}
\end{theorem}

\begin{remark}
When $G$ is equal rank, Eq. (\ref{thm:genrank}) was proved by Connes-Moscovici \cite{Connes-Moscovici:L2}, i.e. the $L^2$-trace of the operator $D$ is the formal degree of the associated discrete series representation of $G$. 
\end{remark}
\begin{remark}
When the kernel of $D$ gives a discrete series representation, i.e. $G$ has equal rank, Equation (\ref{eq:char}) in Theorem \ref{thm:st-pairing}  may be derived from Theorem \ref{thm:HST} and the computation in Hochs and Wang \cite{Hochs-Wang-KT}. Here, we do not assume $G$ to have equal rank in Theorem \ref{thm:st-pairing} and allow the kernel of $D$ to be a limit of discrete series representation. Equation (\ref{eq:char}) suggests that the index pairing can be used to detect some of the limits of discrete series representations of $G$.  
\end{remark}

\begin{remark}In Theorem \ref{thm:genrank}, we have assumed $P$ to be maximal. When $P$ is not maximal, it can be shown that the pairing in Theorem \ref{thm:st-pairing} vanishes. 
\end{remark}
As a special case of Theorem \ref{thm:PPT-index}, we have the following theorem generalizing the Connes-Moscovici $L^2$-index theorem, Theorem \ref{thm:con-mos-L2}, for homogeneous spaces.  
\begin{theorem}\label{thm:gk-vanest} Suppose that $G$ is unimodular. For a $G$-invariant Dirac operator on $X=G/K$
and $[\varphi]\in H^{2k}_{\rm diff}(G)$, we have
\[
\chi(\varphi)(\Ind_{C^*_r(G)}(D_\mu))=\frac{1}{(2\pi \sqrt{-1})^k}\Big\langle \widehat{A}(\mathfrak{g}, K)\wedge \operatorname{ch}(V_\mu)_{\mathfrak{m}^*}\wedge \Phi([\varphi]), [V]\Big\rangle,
\]
where $\Phi([\varphi])$ is a class in $H^{2k}(\mathfrak{g}, K)$ defined by a $G$-invariant closed differential form on $G/K$ via the van Est isomorphism in Proposition \ref{prop:vanest}. 
\end{theorem}
\subsection{Summary of results for manifolds without boundary}
We have introduced 0-cyclic cocycles  $\tr_e$ and $\tr_g$ on the Harish-Chandra algebra $\mathcal{C}(G)$
and we have stated index theorems computing  the pairing of these 0-cyclic cocycles with the
index class $\operatorname{Ind}_{\mathcal{C}(G)}(D)\in K_0 (\mathcal{C}(G))
 \equiv \operatorname{Ind}_{C^*_r(G)}(D)\in  K_0 (C^*(G))$:
 $$\tr_e (\operatorname{Ind}_{\mathcal{C}(G)}(D))\,,\quad 
 \tr_g (\operatorname{Ind}_{\mathcal{C}(G)}(D))$$
  These results are due to Wang, c.f. Theorem \ref{thm:HWang-L2},  and Hochs-Wang, c.f. Theorem \ref{thm:HochsWang-HC}.
 We have then stated generalizations of these two theorems: the first theorem, c.f. Theorem \ref{thm:PPT-index}, by Pflaum-Posthuma-Tang, computes the pairing 
 of the index class with the cyclic cocycles 
 associated to cocycles in the differentiable group cohomology introduced in Theorem \ref{thm:character}; the second theorem, c.f. Theorem \ref{thm:HST} by Hochs-Song-Tang, computes the pairing between the index class and the higher orbital integrals introduced in Theorem \ref{thm:higherorbital}.\\

\section{Proper cocompact $G$-manifolds with boundaries}
\label{sec:boundary}

 In this section, we introduce the index class for a Dirac operator on a proper cocompact $G$-manifold.

\subsection{Geometry on $G$ proper manifolds with boundary}
We start with some generalities.
Let $Y_0$ be a  manifold with boundary, $G$ a finitely connected Lie group acting properly
and cocompactly on $Y_0$. We denote by $X$ the boundary of $Y_0$.
There exists a collar neighbourhood $U$ of the boundary $\pa Y_0$, $U\cong [0,2]\times \partial Y_0$,
 which is $G$-invariant and such that the action of $G$ on $U$ is of product type.
We assume that $Y_0$ is endowed with a $G$-invariant metric $\mathbf{h}_0$ which is of product type near the boundary. 
 We let
$(Y_0,\mathbf{h}_0)$ be the resulting Riemannian  manifold with boundary;  in the collar
neighborhood $U\cong [0,2]\times \partial Y_0$ the metric  $\mathbf{h}_0$  can be written, through the above isomorphism,
 as $dt^2 + \mathbf{h}_X$,
with $\mathbf{h}_X $ a $G$-invariant Riemannian metric on  $X=\partial Y_0$. We denote by $c_{Y_0}$ 
a cut-off function for the action of $G$ on $Y_0$; since the action is cocompact, this is a compactly supported smooth function.
We consider the associated manifold with cylindrical ends
$\widehat{Y}:= Y_0\cup_{\partial Y_0} \left(   (-\infty,0] \times \partial Y_0 \right)$,
endowed with the extended metric $\widehat{\mathbf{h}}$ and the extended $G$-action.
We denote by  $(Y,\mathbf{h})$ the $b$-manifold associated to  $(\widehat{Y},\widehat{\mathbf{h}})$. See \cite{Melrose-Book}. We shall often treat 
$(\widehat{Y},\widehat{\mathbf{h}})$ and $(Y,\mathbf{h})$ as the same object. We denote by $c_Y$ the  obvious extension of the cut-off function
$c_{Y_0}$ for the action of $G$ on $Y_0$ (constant along the cylindrical end); this is a cut-off function of the extended action of $G$ on $Y$.
If $x$ is a boundary defining function for the cocompact $G$-manifold $Y_0$, then the $b$-metric $\mathbf{h}$ has the following
product-structure near the boundary $X$:
 $$\frac{dx^2}{x^2}+ \mathbf{h}_X.$$

\medskip
Let us  fix a slice $Z_0$ for the $G$ action on $Y_0$; thus
$$Y_0\cong G\times_K Z_0$$
with $K$ a maximal compact subgroup of $G$ and $Z_0$ a smooth compact manifold  with boundary endowed with a $K$-action. 
We denote by $S$ the boundary $\partial Z_0$. Consequently, $Y\cong G\times_K Z$
with $Z$ the $b$-manifold associated to $Z_0$ and the boundary
\[
X = \partial Y_0 \cong G \times_K S. 
\] 

We shall concentrate on the case where the symmetric space $G/K$, the $G$-manifold $Y_0$ and the $K$-slice $Z_0$
are all even dimensional. 

\begin{example}
\label{basic-example-1}
Start with an inclusion $K\subset G$ of Lie groups with $K$ compact, and 
let $Z_0$ be a compact $K$-manifold with boundary $\partial Z_0 =:S_0$. Then $Y_0:= G\times_K Z_0$ is an example of manifold with boundary $\partial Y_0=G\times_K\partial S_0$, equipped with a proper, cocompact action of $G$. We also have the associated 
$b$-manifolds $Z$ and  $Y=G\times_K Z$.
If we choose a $K$-
invariant inner product on the Lie algebra $\mathfrak{g}$ of $G$
and a $K$-invariant $b$-metric $g_{Z}$ on $Z$ of product-type near the boundary
then, as in the closed case, we obtain a $G$-invariant metric on $Y$; we call such a metric  {\it slice compatible}. 

As an explicit example, consider $G=SL(2,\mathbb{R})$ and $K=SO(2)$ acting on the unit disk $\mathbb{D}^2$ in the complex plane by rotations around the origin. The resulting manifold 
$Y:=SL(2,\mathbb{R})\times_{SO(2)}\mathbb{D}^2$ is a $4$-dimensional fiber bundle over hyperbolic $2$-space $SL(2,\mathbb{R})\slash SO(2)\cong\mathbb{H}^2$ with fiber $\mathbb{D}^2$. 
The boundary $\partial Y$ of this manifold is isomorphic to $SL(2,\mathbb{R})$.
\end{example}

\subsection{Dirac operators}\label{subsect:b-dirac}$\;$\\
We assume the existence of  a $G$-equivariant bundle of Clifford modules $E_0$ on $Y_0$, endowed with a Hermitian
metric, product-type near the boundary, for which the Clifford action is
unitary, and  equipped with a Clifford connection also of product type near the boundary. 
Associated to these structures there is a generalized $G$-invariant 
Dirac operator $D$ on $Y_0$ with product structure near the
boundary acting on the sections of  $E_0$. We denote by $D_{\pa}$ the operator induced on the boundary.
We employ the same symbol, $D$, for the associated $b$-Dirac operator on $M$, acting on the extended
Clifford module $E$.
We also have $D_{\cyl}$ on $\RR\times \pa Y_0\equiv \cyl(\pa Y_0)$.
We shall make the following fundamental assumption.
\begin{assumption}\label{assumption:invertibility}
There exists  $\alpha>0$ such that
\begin{equation}\label{invertibility}
{\rm spec}_{L^2} (D_{\partial})\cap [-\alpha,\alpha]=\emptyset.
\end{equation}
\end{assumption}
We also have $D_{\cyl}$ on $\RR\times \pa Y_0\equiv \cyl(\pa Y_0)$. It should be noticed that because of the self-adjointness of $D_{\partial}$,
assumption \eqref{invertibility} implies the $L^2$-invertibility of $D_{\cyl}$.

\begin{example}
As an example where this condition is satisfied we can consider a $G$-proper manifold with boundary
with a $G$-invariant riemannian metric and a $G$-invariant Spin structure with the property that the metric
on the boundary is of positive scalar curvature. We would then consider the Spin-Dirac operator $D$;
because of the psc assumption on the boundary we do have that $D_\partial$ is $L^2$-invertible.
These manifolds arise as in \cite{guo-mathai-wang-psc} from the slice theorem and  a $K$-Spin-manifold with boundary $S$
endowed with a $K$-invariant metric which is of psc on $\partial S$. Such compact manifolds $S$ arise, for example,  as follows.
Consider a compact $K$-manifold without boundary $N$ endowed with a $K$-invariant metric of positive scalar curvature.
For the existence of such manifolds see for example \cite{lawson-yau-psc, hanke-symmetry, wiemeler-TAMS}. We can 
now perform on this manifold $K$-equivariant surgeries and produce  along the process a $K$-manifold with boundary $W$.
Under suitable conditions (for example, equivariant surgeries only of codimension
at least equal to 3) this manifold with boundary $W$ will have a $K$-invariant metric of positive scalar curvature. 
We can now take the connected sum of $W$  with a closed $K$-manifold not admitting a K-invariant metric of psc. See \cite{hanke-symmetry, wiemeler-TAMS}. The result will be a manifold $S$ with a K-invariant metric which is of psc
(only) on  $\partial S$.
\end{example}

\subsection{The index class and $b$-calculus}

Let $Y_0$ be a $G$-proper manifold with boundary, with compact quotient, and let
$Y$ be the associated manifold with cylindrical ends, or, equivalently, the associated $b$-manifold.
In the $b$-picture we consider $\mathcal{E}_b$, the $C^*_r G$-Hilbert module obtained
 by (double) completion of  $\dot{C}^\infty_c (Y,E)$ endowed with the $C^\infty_c (G)$-valued
inner product 
$$(e,e^\prime)_{C^*_c (G)} (x):= (e,x\cdot e^\prime)_{L^2_b (Y,E)}, \quad e, e^\prime
\in \dot{C}^\infty_c (Y,E), \;x\in G\,.$$
Here the dot means vanishing of infinite order at the boundary; notice that on the right hand side we employ
$L^2_b (Y,E)$, the $L^2$ space associated to the volume form associated to the $b$ metric $\mathbf{h}$.
One can prove that   $\mathbb{K}(\mathcal{E}_b)$ is isomorphic to the 
relative Roe algebra $C^* (Y_0 \subset Y,E)^G$, 
the latter defined by completion of operators with propagation at a finite distance from $Y_0$.
See \cite{PS-Stolz} for precise definitions and proofs.
We then have the following canonical
isomorphisms
\begin{equation}\label{iso-in-K}
K_* (\mathbb{K}(\mathcal{E}_b))\cong  K_* (C^* (Y_0\subset Y,E)^G)\cong K_* (C^* (Y_0,E_0)^G)\cong K_* (C_r^*(G))
\end{equation}
with the second isomorphism also explained in \cite{PS-Stolz} and the third one the Morita isomorphism already explained in these notes.

We now want to prove the existence  of an index class $\Ind (D)\in K_0 (C^*(Y\subset Y,E)^G)$ under Assumption \ref{assumption:invertibility};
next, we shall want to
 find a smooth, i.e. dense and holomorphically closed, subalgebra $\mathcal{A}^\infty (Y,E)$ of $C^* (Y_0\subset Y,E)^G$ and a
smooth representative $\Ind_\infty (D)\in K_0 (\mathcal{A}^\infty (Y,E))$ corresponding  to  $\Ind (D)\in K_0 (C^*(Y\subset Y,E)^G)$
under the natural isomorphism between $K_0 (\mathcal{A}^\infty (Y,E))$ and $K_0 (C^*(Y\subset Y,E)^G)$.\\
We use the $b$-pseudodifferential
calculus as in \cite{PP2}, even though this is not strictly necessary for the first task. 
For a detailed account of the $b$-calculus we refer the reader to Melrose book, \cite{Melrose-Book}.

\noindent
Let us recall very briefly the basics of the $b$-calculus.
For simplicity we assume that the boundary of our cocompact $G$-proper manifold $Y_0$ is connected. As above, we denote by $Y$
the associated $b$-manifold.
We sometimes denote the boundary of $Y_0$ by $X$; for notational convenience we set $\pa Y=X$. Finally, we often expunge the vector
bundle $E$ from the notation.

\medskip
We can define the $b$-stretched product $Y\times_b Y$ which inherits in a natural way an action of
$G\times G$ and a diagonal action of $G$, see \cite{LeichtnamPiazzaMemoires}. Proceeding as in these references, we can define
the algebra of $G$-equivariant $b$-pseudodifferential operators on $Y$ with $G$-compact support, denoted ${}^b \Psi^*_{G,c}(Y)$, which is a $\ZZ$-graded algebra.  

Let us fix $\epsilon >0$. Then, as in the compact case, we can extend  the algebra ${}^b \Psi^*_{G,c}(Y)$ and consider
 the $G$-equivariant $b$-calculus of $G$-compact support 
with $\epsilon$-bounds, denoted  ${}^b \Psi^{*,\epsilon}_{G,c}(Y)$. Thus, by definition,
$${}^b \Psi^{*,\epsilon}_{G,c}(Y)= {}^b \Psi^{*}_{G,c}(Y) + {}^b \Psi^{-\infty,\epsilon}_{G,c}(Y) + \Psi^{-\infty,\epsilon}_{G,c}(Y).$$
The second term on the right hand side corresponds to the 
Schwartz kernels on $Y\times_b Y$, smooth in the interior,
conormal of order $\epsilon$ on the left and right boundaries  of $Y\times_b Y$
and smooth up to the front face (and of course $G$-equivariant). The third term on the right hand side 
corresponds instead to  Schwartz kernels on $Y\times Y$, smooth in the interior and
conormal of order $\epsilon$ on the boundary hypersurfaces of $Y\times Y$. See \cite{Melrose-Book} and for a quick treatment 
the Appendix in \cite{Mel-P1}.  Elements in 
$\Psi^{-\infty,\epsilon}_{G,c}(Y)$ are called {\it residual}.
Composition formulae 
for the elements in the calculus with bounds are  as in \cite[Theorem 4]{Mel-P1}. 

\smallskip
\noindent
Restriction to the front face of 
$Y\times_b Y$ defines  the indicial operator $I\,:\, {}^b \Psi^{-\infty}_{G,c} (Y)\to {}^b \Psi^{-\infty}_{G,c,\RR^+} (\overline{N^+ \pa Y})$ and, more generally, 
 \begin{equation}\label{indicial}I\,:\, {}^b \Psi^{-\infty,\epsilon}_{G,c} (Y)\to {}^b \Psi^{-\infty,\epsilon}_{G,c,\RR^+} (\overline{N^+ \pa Y})
 \end{equation}
 with $\overline{N^+ \pa Y}$ denoting the compactified inward-pointing normal bundle to the boundary and the subscript
$\RR^+$ denoting equivariance with respect to the natural $\RR^+$-action.  See \cite[Section 4.15]{Melrose-Book},
where $\overline{N^+ \pa Y}$ is denoted $\widetilde{Y}$.
Recall that there is a diffeomorphism of $b$-manifolds:  $\overline{N^+ \pa Y}\simeq \pa Y\times [-1,1]$. Metrically we can think of
$\overline{N^+ \pa Y}$ as the cylinder $\cyl(\pa Y):=\RR\times \pa Y$ and with a small abuse  of notation
we shall adopt this notation henceforth, thus writing
\eqref{indicial}
as
 \begin{equation}\label{indicial-bis}I\,:\, {}^b \Psi^{-\infty,\epsilon}_{G,c} (Y)\to {}^b \Psi^{-\infty,\epsilon}_{G,c,\RR^+} (\cyl(\pa Y)).
 \end{equation}

By fixing a cut-off function $\chi$ on the collar neighborhood of the
boundary, equal to 1 on the boundary, we can define a section $s$ to the
indicial homomorphism $I$: 
\begin{equation}\label{map-s}
s:   {}^b \Psi^{-\infty,\epsilon}_{c,G,\RR^+} (\cyl(\pa Y))\to  {}^b \Psi^{-\infty,\epsilon}_{G,c} (Y)\,;
\end{equation} 
the linear map $s$ associates to an $\RR^+$-invariant operator
$G$ on  $\cyl(\pa Y)$ an operator on the $b$-manifold  $Y$; the latter is obtained by pre-multiplying
and post-multiplying by the cut-off function $\chi$.\\
The Mellin transform in the $s$-variable for  an $\RR^+$-invariant kernel $\kappa (s,y,y')\in 
 {}^b \Psi^{-\infty,\epsilon}_{G,c,\RR^+} (\cyl(\pa Y))$ defines an isomorphism $\mathcal{M}$
between ${}^b \Psi^{-\infty,\epsilon}_{G,c,\RR^+} (\cyl(\pa Y))$ and holomorphic families 
of operators 
$$\{\RR\times i (-\epsilon,\epsilon)\ni\lambda \to \Psi^{-\infty}_{G,c} (\partial Y)\}$$
which rapidly decrease as $|{\rm Re}\lambda |\to \infty$  as functions with values in the Fr\'echet algebras $\Psi^{-\infty}_{c,G} (\partial Y)$. We denote by $I(R,\lambda)$ or sometimes by $\widehat{R}(\lambda)$, the Mellin transform of $R\in  
{}^b \Psi^{-\infty,\epsilon}_{c,G,\RR^+} (\cyl(\pa Y))$; if $P\in  {}^b \Psi^{-\infty,\epsilon}_{G,c} (Y)$
then its indicial family is by definition the indicial family associated to its indicial operator $I(P)$. The inverse $\mathcal{M}^{-1}$ is obtained by associating to a holomorphic 
family $\{S(\lambda), |{\rm Im} \lambda|<\epsilon \}$, rapidly decreasing in ${\rm Re}\lambda$,
the $\RR^+$-invariant Schwartz kernel $\kappa_S$ that in projective coordinates is given by 
\begin{equation}\label{inverse-mellin}
\kappa_G (s,y,y^\prime)=\int_{{\rm Im}\lambda=r} s^{i\lambda} G(\lambda)(y,y^\prime)d\lambda
\end{equation}
with $r\in (-\epsilon,\epsilon)$.

\medskip
Let $D$ be an equivariant  $\ZZ_2$-graded odd  Dirac operator, of product type near the boundary and let us make Assumption \ref{assumption:invertibility}.
We sketch the proof  of  the existence of a $C^*$-index class associated to $D^+$.
One begins
by finding a symbolic parametrix $Q_\sigma$ to $D^+$, with remainders $R^\pm_\sigma$:
\begin{equation}\label{para1}
Q_\sigma D^+=\Id - R^+_\sigma\,,\quad D^+ Q_\sigma=\Id - R_\sigma^-\,.
\end{equation}
We can choose
$Q_\sigma\in {}^b \Psi^{-1}_{G,c}(Y)$, i.e. of $G$-compact support
and then get
 $R^\pm_\sigma\in {}^b \Psi^{-\infty}_{G,c}(Y)$.

Consider now the indicial operator $I(D^+)\equiv D^+_{{\rm cyl}}\equiv \partial/\partial t + D_{\partial}$; by Assumption  \ref{assumption:invertibility} we know that
this operator is $L^2$-invertible on the $b$-cylinder ${\rm cyl}(\partial Y)$. 
Consider the 
smooth 
 kernel $K(Q^\prime)$ on $Y\times_b Y$ which is zero outside a neighbourhood $U_{{\rm ff}}$ of the front face 
 and such that in $U_{{\rm ff}}$ 
 with (projective) coordinates $(x,s,y,y')$ is equal to 
$$K(Q^\prime)(s,x,y,y^\prime):= \chi (x) \int_{-\infty}^\infty s^{i\lambda} K((D_\pa +i\lambda)^{-1}  \circ I( R_\sigma^-,\lambda))(y,y^\prime) d\lambda\,.$$
Put it differently, $Q^\prime= s(I(D)^{-1}I(R_\sigma)$, with $s$ as in \eqref{map-s} and where we have used the inverse Mellin transform 
for the expression of $ I(D)^{-1}I(R_\sigma)$.
Notice that because of the presence of $(D_\pa +i\lambda)^{-1} $ this is not compactly supported. Still,
by proceeding {\it exactly} as in \cite[Lemma 4 and 5]{LPETALE} we can establish the following fundamental result:

\begin{theorem}\label{b-index-class-C*}
The kernel $K(Q^\prime)$ defines  a bounded operator on the $C^*_r G$-module $\mathcal{E}_b$:
$Q^\prime\in \mathbb{B}(\mathcal{E}_b)$.  If $Q^b:= Q_\sigma-Q^\prime$, then $Q^b$ provides an inverse of $D^+$ modulo elements in $\mathbb{K}(\mathcal{E}_b)$. Thus, for a Dirac operator $D$ satisfying the Assumption \ref{assumption:invertibility} there is a well defined index class $$\Ind_{C^*} (D)\in K_0 ( \mathbb{K}(\mathcal{E}_b))\equiv K_0 (C^* (Y_0\subset Y,E)^G)\,.$$ 
\end{theorem}

\smallskip
\noindent
We sometime refer to the passage from $Q_\sigma$ to $Q^b=Q_\sigma-Q^\prime$ as the passage from a symbolic parametrix to a {\it true parametrix}
or an {\it improved parametrix}.

\smallskip

The $b$-calculus approach to the $C^*$ index class will prove to be useful 
in establishing  the existence of a {\it smooth} index class, see below, and for proving a {\it higher} APS index formula.
However, there are other approaches to the $C^*$-index class. Indeed, in the recent
article  \cite{HWW2}  a coarse approach is explained. It should not be difficult
to show that the two index classes, the one defined here and the one defined in  \cite{HWW2}  are in fact the same.
(The proof should proceed as for  Galois coverings, where the analogous result is established in
\cite[Proposition 2.4]{PS-Stolz}.)\\
As in the case of Galois coverings, and always under the invertibility Assumption \ref{assumption:invertibility},
it should also be possible to establish the existence of  an APS $C^*$-index class through
a boundary value problem defined through the spectral projection $\chi_{(0,\infty)} (D_\partial)$
and show that it is equal to the index class as defined above. See \cite{LP-JFA} for the case of Galois coverings.

\medskip
Assume now that $G$ is connected, or more generally that $| \pi_0 (G)|<\infty$; we can apply to $Y$ the slice theorem and obtain a diffeomorphism
$$ G\times_K Z \xrightarrow{\alpha} Y\,$$
with $Z$ a  compact manifold with boundary.
One can prove, as in the closed case, that
$${}^b \Psi^{-\infty,\epsilon}_{G,c}(Y)= (C^\infty_c (G)\hat{\otimes}\, {}^b\Psi^{-\infty,\epsilon}(Z))^{K\times K}.$$
 A similar description can be given for
$\Psi^{-\infty,\epsilon}_{G,c}(Y)$.\\
We set
$${}^b \mathcal{A}^{c,\epsilon}_{G}(Y):={}^b \Psi^{-\infty,\epsilon}_{G,c}(Y)+
\Psi^{-\infty,\epsilon}_{G,c}(Y)$$
Similarly, we  consider
\begin{equation} 
{}^b \mathcal{A}^{c,\epsilon}_{G,\RR^+} (\cyl(\pa Y)):= 
C^\infty_c (G)\hat{\otimes} \,{}^b \Psi^{-\infty,\epsilon}_{\RR^+}(\cyl(\pa Z))^{K\times K}
\end{equation}
and observe that
the indicial operator  induces a surjective  algebra homomorphism
$${}^b \mathcal{A}^{c,\epsilon} _G (Y)\xrightarrow{I} {}^b \mathcal{A}^{c,\epsilon} _{G,\RR^+} (\cyl(\pa Y))$$
where $I$ is equal to the indicial operator on 
${}^b \Psi^{-\infty,\epsilon}_{G,c}(Y)=
(C^\infty_c G)\hat{\otimes} \,{}^b \Psi^{-\infty,\epsilon}(Z))^{K\times K}$
and it is defined as zero on $\Psi^{-\infty,\epsilon}_{G,c}(Y)=
(C^\infty_c G)\hat{\otimes}  \,\Psi^{-\infty,\epsilon}(Z))^{K\times K}$.

$\,$

\smallskip
\noindent
{\bf Notation:} from now on we shall omit the $\epsilon$ from the notations for $^b\mathcal{A}^{c,\epsilon}_G(Y)$ and other related algebras.

\smallskip
\noindent

We obtain in this way algebras of operators on $Y$ fitting into a short exact sequence
\begin{equation}\label{b-short-c}0\to  \mathcal{A}^c_G (Y)\to 
{}^b \mathcal{A}^c_G (Y)\xrightarrow{I} {}^b \mathcal{A}^c_{G,\RR} ({\rm cyl}(\partial Y))\to 0.
\end{equation}
Here \begin{equation}\label{ic-b-residual}
 \mathcal{A}^c_G (Y):= \Ker ({}^b \mathcal{A}^c_G (Y)\xrightarrow{I} {}^b \mathcal{A}^c_{G,\RR} ({\rm cyl}(\partial Y))).
\end{equation}

\subsection{Harish-Chandra smoothing operators}
Recall that 
$${}^b \mathcal{A}^c_G (Y)=(C^\infty_c (G)\hat{\otimes}\, {}^b\Psi^{-\infty,\epsilon}(Z))^{K\times K} + (C^\infty_c (G)\hat{\otimes}\, \Psi^{-\infty,\epsilon}(Z))^{K\times K},$$ $$
{}^b \mathcal{A}^c_{G,\RR} ({\rm cyl}(\partial Y))=(C^\infty_c (G)\hat{\otimes}\, {}^b \Psi^{-\infty,\epsilon}_{\RR} (\cyl (\partial Z)))^{K\times K}.$$
Employing the Harish-Chandra algebra  $\mathcal{C}(G)$ instead of $C^\infty_c (G)$ on the right hand side
we can define 
 the algebras  
 ${}^b \mathcal{A}^\infty_G (Y)$ and ${}^b \mathcal{A}^\infty_{G,\RR} ({\rm cyl}(\partial Y))$
fitting into the short exact sequence
\begin{equation}\label{b-short-HC}0\to  \mathcal{A}^\infty_G (Y)\to {}^b \mathcal{A}^\infty_G (Y)\xrightarrow{I} {}^b \mathcal{A}^\infty_{G,\RR} ({\rm cyl}(\partial Y))\to 0\end{equation}
with $\mathcal{A}^\infty_G (Y)={\rm Ker} ({}^b \mathcal{A}^\infty_G (Y)\xrightarrow{I} {}^b \mathcal{A}^\infty_{G,\RR} ({\rm cyl}(\partial Y))).$
This extends \eqref{b-short-c} and will play a crucial role in our analysis. 
Similar short exact sequences can be defined when the operators act on  sections of a vector bundle $E$.
One can prove, see \cite{PP2}, the following result:

\begin{proposition}
The algebra  $ \mathcal{A}^\infty_G (Y,E)$ is a dense and holomorphically closed subalgebra of the Roe algebra 
$C^*(Y_0\subset Y,E)^G$. Consequently, $K_*( \mathcal{A}^\infty_G (Y,E))\simeq K_* (C^*(Y_0\subset Y,E)^G)$.
\end{proposition}

\subsection{The smooth index class}\label{subsect:smooth-index}
Bearing in mind the last proposition and analyzing in great detail the Schwartz kernels of the operator involved,
one can prove the following crucial results, see  \cite{PP2}. 

\begin{theorem}\label{theo:smooth-index-CS}
Let $D$ be as above and let $Q^\sigma$ be a symbolic $b$-parametrix for $D$. The Connes-Skandalis projector
\begin{equation}\label{CS-projector-true}
P^b_{Q}:= \left(\begin{array}{cc} {}^b S_{+}^2 & {}^b S_{+}  (I+{}^b S_{+}) Q^b \\ {}^b S_{-} D^+ &
I-{}^b S_{-}^2 \end{array} \right)
\end{equation}
associated to a true parametrix $Q^b=Q^\sigma-Q^\prime$ with remainders ${}^b S^\pm$ in $\mathcal{A}^\infty_G (Y_0)$ is  a $2\times 2$ matrix with entries\footnote{As usual, the entries really live in a slightly extended algebra because of the identity appearing in the right lower corner of the matrix.} in $\mathcal{A}^\infty_G (Y)$. We thus have a well-defined {\it smooth} index class
\begin{equation}\label{CS-class-bis}
\Ind_\infty (D):=[P^b_{Q}] - [e_1]\in K_0(\mathcal{A}^\infty_G (Y))\equiv 
K_0 (C^*(Y_0\subset Y)^G)
\;\;\;\text{with}\;\;\;e_1:=\left( \begin{array}{cc} 0 & 0 \\ 0&1
\end{array} \right).
\end{equation}
\end{theorem}

We can consider in particular the parametrix 
$$ Q_{\exp}
:= \frac{I-\exp(-\frac{1}{2} D^- D^+)}{D^- D^+} D^+$$ and its associated true parametrix $Q^b_{\exp}$. 
Using the first, we can define the Connes-Moscovici projector $V_{{\rm CM}} (D)$; using the latter we define the {\it improved}
Connes-Moscovici projector $V^b_{{\rm CM}} (D)$. We then have:

\begin{theorem}\label{theo:smooth-index-CM}
The projector $V^b_{{\rm CM}} (D)$ obtained from the true parametrix $Q^b_{\exp}$ associated to $Q_{\exp}$
is a $2\times 2$ matrix with entries
in $\mathcal{A}^\infty_G (Y)$ and defines the same smooth index class in  $K_0(\mathcal{A}^\infty_G (Y))$
  as the Connes-Skandalis
projector of Theorem \ref{theo:smooth-index-CS}.
\end{theorem}
We can of course extend these results to the associated symmetrized projectors.

\subsection{Higher APS-indices}

We now want to extract numbers out of our index class $\Ind_\infty (D)\in K_0(\mathcal{A}^\infty_G (Y))=K_0 (C^* (Y_0\subset Y)^G)$
and we shall do so by pairing the index class with the cyclic cocycles already encountered in the closed case, that is

\begin{itemize}
\item the 0-cocycle ${\rm tr}^Y_g$, $g\in G$, with $g=e$ or $g\not= e$ a semisimple element,
\item the higher cocycles $\tau^Y_\varphi$, with $\varphi\in H^*_{{\rm diff}} (G)$,
\item the higher delocalized cocycles $\Phi^P_{Y,g}$, with $P$ a cuspidal parabolic subgroup and $g$ a semisimple element.
\end{itemize}
These cocycles do define elements in $HP^* (\mathcal{A}^c_G (Y,E))$, because the Schwartz kernels of the operators
in $\mathcal{A}^c_G (Y,E)$ vanish on all of the boundary hypersurfaces of $Y\times_b Y$. 
The following proposition is proved as in the closed case, using the information that the corresponding cyclic cocycles on the group $G$, that is
$${\rm tr}_g\,,\quad \chi(\varphi)\,,\quad \Phi^P_{g},$$
do extend from $C^\infty_c (G)$ to $\mathcal{C}(G)$.

\begin{proposition}\label{extend}
The cyclic cocycles 
\begin{equation}\label{3cocycles-on-Y}
{\rm tr}^Y_g\,,\quad \tau^Y_\varphi\,,\quad \Phi^P_{Y,g}
\end{equation}
extend continuously from  $\mathcal{A}^c_G (Y,E)$ to $\mathcal{A}^\infty_G (Y,E)$, thus defining elements in the periodic cyclic cohomology $HP^{\text{even}} (\mathcal{A}^\infty_G (Y,E))$.
\end{proposition}
Using the pairing between $HP^{{\rm even}} (\mathcal{A}^\infty_G (Y,E))$ and $K_0 (\mathcal{A}^\infty_G (Y,E))$ we obtain the
numerical indices
\begin{equation}\label{indices}
\langle {\rm tr}^Y_g, \Ind_\infty (D)\rangle\,,\quad 
\langle \tau^Y_\varphi, \Ind_\infty (D)\rangle\,,\quad \langle \Phi^P_{Y,g}, \Ind_\infty (D)\rangle\,.\end{equation}
In the next section we shall present formulae \`a la Atiyah-Patodi-Singer for these (higher) indices.

\section{Higher APS index theorems}
\label{sec:aps index}

In this section we shall finally present index formulae of Atiyah-Patodi-Singer type for the
numerical indices appearing in \eqref{indices}.
We shall employ in a crucial way  relative K-theory and relative cyclic cohomology, as in previous work of Moriyoshi-Piazza,
Lesch-Moscovici-Pflaum and Gorokhovsky-Moriyoshi-Piazza, see \cite{moriyoshi-piazza, LMP, GMPi}. To be more precise,
and forgetting the bundle $E$ in the notation, we shall proceed as follows.
Consider  the surjective homomorphism
\begin{equation}\label{su} {}^b \mathcal{A}^\infty_G (Y)\xrightarrow{I} {}^b \mathcal{A}^\infty_{G,\RR^+} 
(\cyl (\pa Y)).
\end{equation}
Then, as a first step, associated to \eqref{su}, we shall define a relative index class 
\begin{equation} \Ind_\infty (D,D_\partial) \in K_0 ({}^b \mathcal{A}^\infty_G (Y), {}^b \mathcal{A}^\infty_{G,\RR^+} 
(\cyl (\pa Y))
\end{equation}
where, with a small abuse of notation, we do not write the indicial homomorphism.\\
Next,
let $\omega^Y$ be one of the cyclic cocycles appearing in \eqref{3cocycles-on-Y}; then  

\begin{itemize}
\item to the cyclic cocycle $\omega^Y$ on $\mathcal{A}^\infty_G (Y)$ we associate a regularised cyclic cochain 
on ${}^b \mathcal{A}^\infty_G (Y)$  and a cyclic cocycle $\xi^{\partial Y}$ on $ {}^b \mathcal{A}^\infty_{G,\RR^+} (\cyl (\pa Y))$ 
so that $(\omega^Y,\xi^{\partial Y})$ is a {\it relative} cyclic cocycle for  the homomorphism \eqref{su};
\item we call $\xi^{\partial Y}$ the eta cocycle associated to $\omega^Y$ and we call $(\omega^Y,\xi^{\partial Y})$ the relative 
cyclic cocycle  associated to $\omega^Y$;
\item we prove that 
$$\langle  \omega^Y, \Ind_\infty (D)\rangle= \langle (\omega^{Y,r},\xi^{\partial Y}), \Ind_\infty (D,D_\pa)\rangle.$$

\end{itemize}
Let us see some of details involved in this program.

\subsection{Relative index classes}\label{subsect:smooth-index-relative}
Recall that if  $0\to J\to A\xrightarrow{\pi} B\to 0$ is a short exact sequence of Fr\'echet
  algebras then we set $K_0 (J):=
K_0 (J^+,J)\cong \Ker (K_0 (J^+)\to \ZZ)$
and $K_0 (A^+,B^+)= K_0 (A,B)$ with $( )^+$ denoting unitalization.
See \cite{bla:book,hr-book}. 
Recall that a relative $K_0$-element
for $ A\xrightarrow{\pi} B$ with unital algebras $ A, B$
is represented by a  triple $(P,Q,p_t)$ with $P$ and $Q$ idempotents in $M_{n\times n} (A)$
and $p_t\in M_{n\times n} (B)$ a
path of idempotents connecting $\pi (P)$ to $\pi (Q)$.
The excision  isomorphism
\begin{equation}\label{excision-general}
\alpha_{{\rm ex}}: K_0 (J)\longrightarrow K_0 (A,B)
\end{equation}
is given by
$\alpha_{{\rm ex}}([(P,Q)])=[(P,Q,{\mathfrak c})]$
with ${\mathfrak c}$ denoting the constant path.

Let us go back to the parametrix
$$ Q_{\exp}
:= \frac{I-\exp(-\frac{1}{2} D^- D^+)}{D^- D^+} D^+$$ and its associated true parametrix $Q^b_{\exp}$. 
Using the latter we have defined the {\it improved}
Connes-Moscovici projector $V^b_{{\rm CM}} (D)$, with entries in $\mathcal{A}^\infty (Y,E)$.
Using the former we can now define the usual  Connes-Moscovici projector $V_{{\rm CM}} (D)$.
We have the following important result:

\begin{theorem}\label{theo:smooth-index-CM-bis}
{\ }
\begin{itemize}
\item[1)] The Connes-Moscovici projector $V_{{\rm CM}}(D)$,  
\[
V_{{\rm CM}}(D):=\left( \begin{array}{cc} e^{-D^- D^+} & e^{-\frac{1}{2}D^- D^+}
\left( \frac{I- e^{-D^- D^+}}{D^- D^+} \right) D^-\\
e^{-\frac{1}{2}D^+ D^-}D^+& I- e^{-D^+ D^-}
\end{array} \right),
\]
is a $2\times 2$ matrix with entries in ${}^b\mathcal{A}^\infty_G (Y)$;
\item[2)]
the Connes-Moscovici projector $V_{{\rm CM}}(D^{\cyl})$ is a $2\times 2$ matrix with entries in ${}^b\mathcal{A}^\infty_{G,\RR} ({\rm cyl}(\partial Y))$. 
\end{itemize}
\end{theorem}

Consider the Connes-Moscovici projections $V_{{\rm CM}} (D)$ and $V_{{\rm CM}} (D_{\cyl})$
associated to $D$ and $D_{\cyl}$.
With $e_1:=\begin{pmatrix} 0&0\\0&1 \end{pmatrix}$,  we have the  triple,
\begin{equation}\label{pre-wassermann-triple-a}
(V_{{\rm CM}} (D), e_1, q_t)
\,, \;\;t\in [1,+\infty]\,,\;\;\text{ with }
q_t:= \begin{cases} V_{{\rm CM}} (t D_{\cyl}),
\;\;\quad\text{if}
\;\;\;t\in [1,+\infty),\\
e_1, \;\;\;\;\;\;\;\;\;\;\;\;\;\,\text{ if }
\;\;t=\infty.
 \end{cases}
\end{equation}

\begin{proposition}\label{prop:relative-indeces}
Under the invertibility Assumption \eqref{invertibility}, the Connes-Mos-covici idempotents $V_{{\rm CM}} (D)$
and  $V_{{\rm CM}} (D_{\cyl})$ define  through formula
  \eqref{pre-wassermann-triple-a}
a relative class in $K_0 ({}^b \mathcal{A}^\infty_G (Y),{}^b \mathcal{A}^\infty _{G,\RR^+}  (\cyl (\pa Y))$,
 the relative K-theory group 
associated to  the surjective homomorphism
\begin{equation*}
{}^b \mathcal{A}^\infty_G (Y)\xrightarrow{I} {}^b \mathcal{A}^\infty_{G,\RR^+} 
(\cyl (\pa Y)).
\end{equation*}
With a small abuse of notation we  denote this class by $[V_{{\rm CM}} (D), e_1, V_{{\rm CM}} (t D_{\cyl})]$.
\end{proposition}

\begin{definition}\label{def:rel-index-smooth}
We define the relative (smooth) index class as
$$\Ind_\infty (D,D_\partial):= [V_{{\rm CM}}  ( D), e_1, V_{{\rm CM}} ( D_{\cyl})]\;\;\in\;\;K_0 ({}^b \mathcal{A}^\infty_G (Y),{}^b \mathcal{A}^\infty _{G,\RR^+}  (\cyl(\pa Y))\,.$$
\end{definition}

Recall now the smooth index class $\Ind_\infty (D)\in K_0 (\mathcal{A}^\infty_G (Y))$ defined through Theorem \ref{theo:smooth-index-CM}. The following result plays a crucial role:

\begin{theorem}\label{theo:excision}
Let $\alpha_{{\rm ex}}$  
be the
excision isomorphism for the short exact sequence $0\to  \mathcal{A}^\infty_G (Y)\to {}^b \mathcal{A}^\infty_G (Y)\xrightarrow{I} {}^b \mathcal{A}^\infty_{G,\RR^+}  (\cyl ( \pa Y))\to 0$.
 Then
\begin{equation}\label{excision-for-cs}
\alpha_{{\rm ex}}(\operatorname{Ind}_\infty (D))=
\operatorname{Ind}_\infty (D,D_{\pa})\,.
\end{equation}
\end{theorem}

\noindent
{\bf We summarize the content of subsection \ref{subsect:smooth-index} and the present subsection \ref{subsect:smooth-index-relative} as follows:}\\ 
{\it using the Connes-Moscovici projector(s) we have proved the existence of smooth index classes 
  \[
  \Ind_\infty (D)\in K_0(\mathcal{A}^\infty_G (Y))\ \text{and}\ \Ind_\infty (D,D_\partial )\in 
K_0 ({}^b \mathcal{A}^\infty_G (Y),{}^b \mathcal{A}^\infty_{G,\RR} ({\rm cyl}(\partial Y))),
\]
with the first one sent to the second one by the excision isomorphism
$\alpha_{{\rm ex}}$.}

\subsection{Relative cyclic cocycles}

Recall that the relative cyclic complex associated to a short exact sequence $0\to J\to A\xrightarrow{\pi}B\to 0$ of algebras
is given by 
\[
CC^k(A,B):=CC^k(A)\oplus CC^{k+1}(B),
\]
equipped with the differential
\[
\left(\begin{matrix} b+B&-\pi^*\\ 0&-(b+B)\end{matrix}\right),
\]
where $b,B$ are the usual Hochschild and cyclic differential and $\pi^*$ denotes the pull-back of functionals through the surjective morphism 
$\pi:A\to B$. In our case, the relevant extension is given, first of all, by 
\begin{equation*}
0\to  \mathcal{A}^c_G (Y)\to {}^b \mathcal{A}_{G}^c (Y)\xrightarrow{I} {}^b \mathcal{A}_{G,\RR^+}^c  (\cyl (Y))\to 0\,,\quad   \mathcal{A}^c_G (Y):= \Ker I,
\end{equation*}
where this is now, for simplicity, the short exact sequence for the small $b$-calculus. 

As in the closed case, given  a global slice $Z\subset Y$, we can view $A\in {}^b \mathcal{A}^c_G (Y)$ as a map $\Phi_A:G\to {}^b\Psi^{-\infty}(Z)$ 
by setting $\Phi_A(g,s_1,s_2):=A(s_1,gs_2)$. Likewise, an element $B\in {}^b \mathcal{A}^c_{G,\RR^+}  (\cyl(\pa Y))$
gives rise, by  Mellin transform, to a map $\Phi_B:G\times \mathbb{R}\to \Psi^{-\infty}(\partial Z)$, which is 
compactly supported on $G$, and rapidly decreasing on $\RR$. In the following, we shall denote by 
\begin{equation}\label{mellin}
\begin{split}
\Phi_A\mapsto \hat{\Phi}_A\;\;\text{the morphism}\;\;& I:{}^b \mathcal{A}_{G}^c (Y)\rightarrow {}^b \mathcal{A}_{G,\RR^+}^c  (\cyl(\pa Y))\;\;\\
&\qquad \qquad\text{followed by the Mellin transform.}
\end{split}
\end{equation}

\subsection{The b-Trace on $G$-proper $b$-manifolds}
Let us construct the correct analogue of the $b$-trace of Melrose in this setting where we have a proper group action. In our geometric setting, a choice of cut-off function $c_{Y_0}$
for the action of $G$ on $Y_{0}$ 
restricts to give a cut-off function $c_{\partial Y_{0}}:=c_{Y_0} |_{\partial Y_{0}}$ for 
the $G$-action on $\partial Y_{0}$. We shall also write briefly $c_{\partial}$.
We consider as usual the associated $b$-manifold $Y$, endowed with a product-type $b$-metric $\mathbf{h}$, so that,
metrically, $Y$ is a manifold with cylindrical ends, and we shall, by a small abuse of notation, write $c_Y$ for the extension of $c_{Y_0}$ on $Y_0$ which is constant
in the cylindrical coordinate.

Using the $b$-integral of Melrose, see \cite{Melrose-Book}, we now define
\begin{definition}
For $A\in {}^b \mathcal{A}_{G}^c(Y)$ its $~^{b}\,G$-trace is defined as
\[
{}^b {\rm Tr}_{c_Y}(A):=\bint
 K_A(x,x)c_Y(x)d{\rm vol}(x).
\]
\end{definition}
Remark that the cut-off function on $Y_0$
has compact support, so the $b$-regularized integral is indeed well defined.
The argument in \cite[Prop. 2.3]{PPT} shows that ${}^b {\rm Tr}_{c_Y}$ is independent of the choice of 
cut-off function $c_Y$. Next, using a simple trick with a family $\{c_{Y,\epsilon}\}_{\epsilon>0}$ of cut-off functions converging to 
the characteristic function on $Z$, we can rewrite
\begin{equation}
\label{b-trace}
{}^b {\rm Tr}_{c_Y}(A)={}^b{\rm Tr}_{Y}\left(\Phi_A(e)\right).
\end{equation}
As in the usual $b$-calculus, ${}^b {\rm Tr}_{c_Y}$ is not a trace, but we have a precise formula for its defect on commutators, directly inspired by Melrose'  
$b$-trace formula :
\begin{lemma}
\label{btrace}
For $A_{1},A_{2}\in  {}^b \mathcal{A}_{G}^c (Y)$, we have
\[
{}^b {\rm Tr}_{c_Y}\left([\Phi_{A_{1}},\Phi_{A_{2}}]\right)=\frac{i}{2\pi}\int_{\mathbb{R}}\int_G{\rm Tr}_{\partial Z}\left(\frac{\partial I(\Phi_{A_1},h^{-1},\lambda)}{\partial \lambda}\circ
I(\Phi_{A_2},h,\lambda)\right)dh d\lambda,
\]
with ${\rm Tr}_{\partial Z}$ denoting the usual functional analytic trace on smoothing operators on closed compact manifolds.
\end{lemma}

\subsection{From absolute to relative cyclic cocycles}
Let $\omega^Y$ be any one of the cyclic cocycles   for $\mathcal{A}^\infty_G(Y)$ appearing in \eqref{3cocycles-on-Y}. We can write 
$$\omega^Y=\omega\circ {\rm Tr}_Z$$
where $\omega$ is the corresponding cyclic cocycle on the group $G$ and where ${\rm Tr}_Z:
\mathcal{A}^\infty_G (Y,E)\to \mathcal{C}(G)$ is the homomorphism of integration along the slice.
The  integrals along the slice
are well defined because the operators in $\mathcal{A}^\infty_G(Y)$ vanish to order $\epsilon$ at all the boundary hypersurfaces of $Y\times_b Y$.
If we pass to ${}^b \mathcal{A}^\infty_G(Y)$ we must replace ordinary integration with $b$-integration, as operators 
in  ${}^b \mathcal{A}^\infty_G(Y)$ do {\it not} vanish on the front face of $Y\times_b Y$; we obtain in this way a {\it regularized}
multilinear functional $\omega^{Y,r}$; this will not give us 
a cyclic cocycle anymore, precisely because of Lemma \ref{btrace}; however, using the exact form of the Lemma \ref{btrace} we will be able to produce
 a cyclic cocycle on ${}^b  \mathcal{A}_{G,\RR^+}^c  (\cyl(\pa Y))$, call it $\xi^{\partial Y}$,
in such a way that the pair $(\omega^{Y,r},\xi^{\pa Y})$ is in fact a {\it relative} cyclic cocycle for
the surjective homomorphism  $ {}^b \mathcal{A}_{G}^c (Y)\xrightarrow{I} {}^b  \mathcal{A}_{G,\RR^+}^c  (\cyl(\pa Y))$.

\subsection{Relative cyclic cocycles associated to orbital integrals.}
Let us see how the general principle put forward in the previous subsection  works in the case $\omega^Y={\rm tr}^Y_g$.
If $Y_0$ is a cocompact G-proper manifold with boundary and $Y$ is the associated $b$-manifold,  then associated to the orbital integral ${\rm tr}_g$ we have the trace-homomorphism
\begin{equation}\label{delocalized-trace-on-M}
{\rm tr}_g^Y: \mathcal{A}^\infty_G (Y)\to\CC\,,
\end{equation} 
given explicitly by 
\begin{equation}\label{tr-g-b}
{\rm tr}_g^Y (\kappa):= \int_{G/Z_g}\int_Y c_Y(hgh^{-1}y) {\rm tr} (hgh^{-1}\kappa (hg^{-1}h^{-1}y,y))dx \,d(hZ).
\end{equation}
Here $dy$ denotes the $b$-density associated to the $b$-metric $\mathbf{h}$ and
the cut-off function $c_{Y_0}$ on $Y_0$ is extended constantly along
the cylinder to define $c_Y$. 
 As already explained, the $Y$ integral converges, given that $\kappa$ vanishes of order $\epsilon>0$ on all the boundary hypersurfaces 
 of $Y\times_b Y$.
 
 Now, let $X$ be a closed manifold equipped with a proper, cocompact action of $G$. We consider  ${\rm cyl}(X)=X\times\RR$
the cylinder over $X$, equipped with the action of $G\times\RR$.  
Using the Mellin  transforms we have an algebra 
homomorphism 
\begin{equation}\label{identification}
{}^b \mathcal{A}_{G,\RR}^c  (\cyl (Y))\longrightarrow  \mathscr{S}(\RR,\mathcal{A}^c_{G}(Y)),\quad A\mapsto \hat{A}.
\end{equation}

\begin{proposition}\label{prop:1-eta-cocycle}
Let $X$ be a cocompact $G$-proper manifold without boundary; for example $X=\partial Y_0$.
Define the following 1-cochain on ${}^b \mathcal{A}^c_{G,\RR} ({\rm cyl}(X))$
\begin{equation}\label{1-eta}
\sigma^X_g (A_0,A_1)=\frac{i}{2\pi}\int_\RR {\rm tr}_g^X ( \partial_\lambda \widehat{A}_0 (\lambda)\circ \widehat{A}_1(\lambda)) d\lambda\,.
\end{equation}
Then $\sigma^X_g (\,,\,)$ is well-defined  and  is a cyclic 1-cocycle.
\end{proposition}

Let $Y$ be now a $b$-manifold and let  $\partial Y$ be its boundary.
Let ${\rm tr}^{Y,r}_g$ be the functional on ${}^b \mathcal{A}^c_{G} (Y)$:
$${\rm tr}^{Y,r}_g (T):= 
 \int_{G/Z_g}\int^b_Y c_Y(hgh^{-1}y) {\rm tr} (hgh^{-1}\kappa (hg^{-1}h^{-1}y,y))dy \,d(hZ)\,.$$
Here $\kappa$ is the kernel of the operator $G$, and Melrose's $b$-integral has been used. This is the regularization of ${\rm tr}^Y_g$ that one needs
to consider when one passes from  $\mathcal{A}^c_{G} (Y)$ to  ${}^b \mathcal{A}^c_{G} (Y)$
(for the time being on kernels of $G$-compact support). 
Observe that $${\rm tr}^{Y,r}_g ={\rm tr}_g \circ {}^b {\rm Tr}_Z$$ with ${}^b {\rm Tr}_Z: {}^b \mathcal{A}^c_{G} (Y)\to C^\infty_c (G)$
denoting $b$-integration along the slice $Z$. More precisely, as in the closed case, we have an isomorphism
\begin{eqnarray*}
&&{}^b \mathcal{A}^c_G (Y)\\
&\cong& \left\{\Phi:G\to {}^b \Psi^{-\infty,\epsilon}(Z),~\mbox{smooth, compactly supported and }
~K\times K~\mbox{invariant}\right\}
\end{eqnarray*}
and ${}^b \Tr_Z$ associates to $\Phi$ the function $G\ni \gamma\to {}^b \Tr (\Phi (\gamma))$.

One can prove the following 

\begin{proposition}\label{prop:-0-relative-cocycle}
The pair 
$({\rm tr}^{Y,r}_g,\sigma^{\partial Y}_g)$ defines  a relative 0-cocycle for 
\[
{}^b \mathcal{A}^c_G (Y)\xrightarrow{I} {}^b \mathcal{A}^c_{G,\RR} ({\rm cyl}(\partial Y)).
\]
Moreover, the 0-degree cyclic cocycle 
$({\rm tr}^{Y,r}_g,\sigma^{\partial Y}_g)$ extends continuously to  a relative 0-cocycle for 
$${}^b \mathcal{A}^\infty_G (Y)\xrightarrow{I} {}^b \mathcal{A}^\infty_{G,\RR} ({\rm cyl}(\partial Y)).$$
Finally, the following formula holds:
\begin{equation}\label{basic-formula-3-} 
 \langle {\rm tr}^Y_g,\Ind_\infty (D)\rangle=\langle ({\rm tr}^{Y,r}_g,\sigma^{\partial Y}_g), \Ind_\infty (D,D_\partial )\rangle\,.
 \end{equation}
 \end{proposition}
 
 \subsection{The delocalized APS index formula on $G$-proper manifolds}
 We shall prove the delocalized  APS index formula using 
 crucially \eqref{basic-formula-3-}.
 On the right hand-side we have the pairing of  the relative cocycle $(\tau^{Y,r}_g,\sigma^{\partial Y}_g)$ with the relative index class
 defined by 
\begin{equation}\label{pre-wassermann-triple}
(V_{{\rm CM}}(D), e_1, q_t)
\,, \;\;t\in [1,+\infty]\,,\;\;\text{ with }
q_t:= \begin{cases} V_{{\rm CM}}(t D_{\cyl})
\;\;\quad\text{if}
\;\;\;t\in [1,+\infty)\\
e_1 \;\;\;\;\;\;\;\;\;\;\;\;\;\,\text{ if }
\;\;t=\infty
 \end{cases}
\end{equation}
and with $e_1:=\begin{pmatrix} 0&0\\0&1 \end{pmatrix}$.\\
By definition of 
relative pairing we have:
\begin{equation}\label{relative-pairing}
\begin{split}
&\langle (\tr^{Y,r}_g,\sigma^{\partial Y}_g), (V_{{\rm CM}}(D), e_1, q_t)\rangle\\
=&
 \tr^{Y,r}_g (e^{-D^- D^+})- \tr^{Y,r}_g (e^{-D^+ D^-})
+ \int_1^\infty \sigma^{\partial Y}_g ([\dot{q}_t,q_t],q_t)dt\,.
\end{split}
\end{equation}
A complicated but totally elementary computation shows that the following 
Proposition holds:
\begin{proposition}\label{prop:from-cocycle-to-eta}
 The term $\int_1^\infty \sigma^{\partial Y}_g ([\dot{q}_t,q_t],q_t)dt$, with 
 $q_t:= V_{{\rm CM}}(t D_{\cyl})$ is equal to
 $$
-\frac{1}{2} \left( \frac{1}{\sqrt{\pi}} \int_1^\infty {\rm tr}_g^{\partial Y} (D_{\partial} \exp (-tD^2_{\partial}) )\frac{dt}{\sqrt{t}}\right)
\,.$$
\end{proposition}

\noindent
Thanks to this Proposition we have that 
\begin{equation}\label{relative-pairing-bis}
\begin{split}
&\langle (\tr^{Y,r}_g,\sigma^{\partial Y}_g), (V(D), e_1, q_t)\rangle\\
=&
 \tr^{Y,r}_g (e^{-D^- D^+})- \tr^{Y,r}_g (e^{-D^+ D^-})
-\frac{1}{2}  \int_1^\infty      \frac{1}{\sqrt{\pi}}\tau^{\partial M}_g (D_{\partial Y} \exp (-tD^2_{\partial Y}) )\frac{dt}{\sqrt{t}}   \,.
\end{split}
\end{equation}
 As a last step we replace $D$ by  $sD$; in the equality
 $$ \langle\tr^Y_g,\Ind_\infty (D)\rangle=\langle (\tr^{Y,r}_g,\sigma^{\partial Y}_g), \Ind_\infty (D,D_\partial )\rangle$$
 the left hand side  $\langle \tau^Y_g, \Ind_\infty (D)\rangle$ remains unchanged
 whereas the right hand side 
  becomes
 $$\tr^{Y,r}_g (e^{-s^2 D^- D^+})- \tr^{Y,r}_g (e^{-s^2 D^+ D^-})-  \frac{1}{2}
 \int_s^\infty   \frac{1}{\sqrt{\pi}} \tr^{\partial Y}_g (D_{\partial Y} \exp (-tD^2_{\partial Y}) \frac{dt}{\sqrt{t}}  \,.$$
 Summarizing, for each $s>0$ we have 
  \begin{eqnarray*}\label{s-equality}
  &&\tr^{Y,r}_g (e^{-s^2 D^- D^+})- \tr^{Y,r}_g (e^{-s^2 D^+ D^-})\\
  &=& \langle \tr^Y_g, \Ind_\infty (D)\rangle +  \frac{1}{2}
 \int_s^\infty   \frac{1}{\sqrt{\pi}} \tr^{\partial Y}_g (D_{\partial Y} \exp (-tD^2_{\partial Y}) \frac{dt}{\sqrt{t}}  \,.
 \end{eqnarray*}
 Now we take the limit as $s\downarrow 0$;
 using Getzler' rescaling in the $b$-context one can prove that 
  $$\lim_{s\downarrow 0} \tr^{Y,r}_g (e^{-s^2 D^- D^+})- \tr^{Y,r}_g (e^{-s^2 D^+ D^-})= 
 \int_{Y_0^g} c^g_{Y_0} {\rm AS}_g (D_0)$$ 
 with $c^g_{Y_0} {\rm AS}_g(D_0)$ defined  in Equation (\ref{eq:X-geom-form}).

 \smallskip
 \noindent
 Assuming this  last result, we can infer  that the limit 
 $$ \lim_{s\downarrow 0} \frac{1}{2}
 \int_s^\infty   \frac{1}{\sqrt{\pi}} \tr^{\partial Y}_g (D_{\partial Y} \exp (-tD^2_{\partial Y}) \frac{dt}{\sqrt{t}}$$
 exists and  equals  
  $$\int_{Y_0^g} c^g_{Y_0^g} {\rm AS}_g (D_0) - \langle \tr^Y_g, \Ind_\infty (D)\rangle\,.$$
 We conclude that the following Theorem holds:
 
 \begin{theorem}\label{theo:0-delocalized-aps} (0-degree delocalized APS)\\
 Let $G$ be a connected, linear real reductive group. Let $g$ be a semisimple element. Let $Y_0$, $Y$, $D$, $D_{\partial Y}$ as above.
 Assume that $D_{\partial Y}$ is $L^2$-invertible.
 Then 
 $$\eta_g (D_{\partial Y}):=\frac{1}{\sqrt{\pi}} \int_0^\infty \tr^{\partial Y}_g (D_{\partial Y} \exp (-tD^2_{\partial Y}) \frac{dt}{\sqrt{t}}$$
 exists and for the pairing of the index class $\Ind(D_\infty)\in K_0(\mathcal{A}^\infty_G (Y))\equiv
 K_0 (C^*(Y_0\subset Y)^G)$ with the $0$-cocycle  ${\rm tr}^Y_g\in HC^0 ((\mathcal{A}^\infty_G (Y)))$ 
 the following delocalized 0-degree APS index formula holds:
 \begin{equation}\label{main-0-degree}
 \langle \tr^Y_g,\Ind_\infty (D) \rangle= \int_{Y_0^g} c_{Y_0^g} {\rm AS}_g (D_0) - \frac{1}{2} \eta_g (D_{\partial Y})\,,
 \end{equation}
where the integrand $c_{Y^g_0} {\rm AS}_g(D_0)$ is defined in the same way as the one in Equation (\ref{eq:X-geom-form}). 
 \end{theorem}
 
 \begin{remark}This result was first discussed in the work of Hochs-Wang-Wang \cite{HWW2}.
 Our treatment, centred around the interplay between absolute and relative cyclic cohomology and the $b$-calculus, is completely different; moreover our treatment
 allows us to  get sharper results compared to \cite{HWW2} in the case of a connected linear real reductive group $G$.
  More precisely, in Theorem \ref{theo:0-delocalized-aps} we only assume that $g$ is a semisimple element of $G$ to obtain the index formula (\ref{main-0-degree}), while in \cite[Theorem 2.1]{HWW2}, the authors require that $G/Z_g$ is compact.
  \end{remark}

\begin{remark}
If we take $g=e$ then we get 
$$\left<{\rm tr}_e^Y, \Ind_{C^*} (D)\right> = \int_Y c_Y {\rm AS}(Y) -\frac{1}{2} \eta_G (D_\partial)=\int_{Y_0} c_{Y_0} {\rm AS}(Y_0) -\frac{1}{2} \eta_G (D_\partial)$$
with 
$$\eta_G (D_\partial)= \frac{2}{\sqrt{\pi}}\int_0^\infty \tr^{\partial Y}_e D_\partial e^{-(tD_\partial)^2}dt.$$
Notice however that  this particular result holds under much more general assumptions on $G$ than the ones
we are currently imposing ($G$ connected reductive linear Lie group). 
Indeed, the pairing of the index class with $ {\rm tr}_e^Y$ is equal to the pairing of the Morita equivalent 
$C^*_r(G)$-index class $\Ind_{C^*_r(G)}(D)$ with the canonical trace $ {\rm tr}_e$ on $C^*G$:
$$ \left<\ {\rm tr}_e^Y, \Ind_{C^*} (D)\right> = \left<{\rm tr}_e, \Ind_{C^*G}(D)\right>.$$ 
Proceeding as  in \cite{Wang:L2},
one  checks that $\left<{\rm tr}_e^Y, \Ind_{C^*G}(D)\right> $ is equal to 
the von Neumann G-index of $D^+$. A formula for this von Neumann index can be proved
in the von Neumann framework
by mimicking the proof of Vaillant for Galois coverings of manifolds with cylindrical
ends \cite{Vaillant-master} (in turn inspired by Melrose'  proof on manifolds with cylindrical ends). 
Thus, assuming only that $G$ is a Lie group
but keeping the $L^2$-invertibility of the boundary operator $D_\partial$, we obtain that 
$\left< {\rm tr}_e^Y, \Ind_{C^*} (D)\right>$ and $\left< {\rm tr}_e, \Ind_{C^*G}(D)\right>$ are well defined and that
$$\left< {\rm tr}_e^Y, \Ind_{C^*} (D), \right> = \left< {\rm tr}_e, \Ind_{C^*G}(D)\right>=\Ind_{{\rm vN}} (D^+)= \int_{Y_0} c_{Y_0} {\rm AS}(Y_0) -\frac{1}{2} \eta_G (D_\partial)$$
with 
$\eta_G (D_\partial)= \frac{2}{\sqrt{\pi}}\int_0^\infty  {\rm tr}_e^{\partial Y} D_\partial e^{-(tD_\partial)^2}dt\,.$\\
This formula is the same as the one appearing in the work of Hochs-Wang-Wang, see \cite{HWW1}.
\end{remark}

\subsection{More on delocalized eta invariants.}
In the previous section we have obtained the well-definedness  of $\eta_g (D_{\partial})$, with $D_\partial$ being $L^2$-invertible, as a byproduct of the proof
 of the delocalized APS index theorem for 0-degree cocycles. In fact, one can  show that  $\eta_g (D)$ is well defined
 on a cocompact
 $G$-proper manifold 
 even if $D$ does not arise  as a boundary operator.
 This is the content of the next 
 theorem, partially discussed also in \cite{HWW1,HWW2}.


\begin{theorem}\label{thm etadefine}
Let $(X,\mathbf{h})$ be a cocompact $G$-proper manifold without boundary endowed with a 
$G$-equivariant metric and let $D$ be an $L^2$-invertible Dirac-type operator of the form (\ref{dec dirac}). Let $g$ be a semisimple element. The integral 
\begin{equation}\label{thm:eta}
\frac{1}{\sqrt{\pi}} \int_0^\infty {\rm tr}^X_g (D\exp (-tD^2) \frac{dt}{\sqrt{t}}
\end{equation}
converges and defines the delocalized eta invariant associated to $D$, $\eta_{g}(D)$.
\end{theorem}
Notice that the proof of this result is rather delicate, both at $t=0$, where results of Zhang  
\cite{Zhangwp} are used, and at $t=+\infty$, where a delicate analysis of the large time behaviour 
of the heat kernel is needed.

We also point out that in a previous version of this survey, based on an earlier version of our work with Posthuma and Song \cite{PPST}, this result was wrongly stated without the invertibility 
assumption; our proof of the convergence at $t=+\infty$ does apply to $D_{{\rm split}}$ but not to $D$, which
is why we do need the invertibility assumption.
%

\subsection{Relative cyclic cocycles associated to smooth group cocycles.} In this subsection we want to see 
how, given a smooth group cocycle 
$\varphi\in Z^k_{\rm diff}(G)$,  we can pass from the cyclic cocycle  $\tau^Y_\varphi$ on
$\mathcal{A}_{G}^c (Y)$ to  a relative cyclic cocycle for
the  surjective homomorphism $ {}^b \mathcal{A}_{G}^c (Y)\xrightarrow{I} {}^b  \mathcal{A}_{G,\RR^+}^c  (\cyl(\pa Y))$.
As we have already explained, the first step is to pass from $\tau^Y_\varphi$ to $\tau^{Y,r}_\varphi$, 
and this is achieved by replacing integrals by $b$-integrals or, equivalently, traces by $b$-traces.
This is what we do in the next definition.
 
\begin{definition}\label{regularized-group}
Let $Y$ be a proper $G$-manifold with boundary, and $\varphi\in Z^k_{\rm diff}(G)$ be a smooth group cocycle.  For $A_0,\ldots, A_k\in 
{}^b \mathcal{A}_{G}^c (Y)$, define
\[
\begin{split}
\tau^{Y,r}_\varphi (A_0,\ldots,A_k):=&\int_{G^{\times k}}{}^b{\rm Tr}_Z
\left(\Phi_{A_0}((g_1\cdots g_k)^{-1})\circ \Phi_{A_1}(g_1)\circ \ldots\circ \Phi_{A_k}(g_k)\right)\\
&\qquad \qquad\qquad 
\varphi (e,g_1,g_1g_2,\ldots,g_1\cdots g_k)dg_1\cdots dg_k.
\end{split}
\]
In the above equation we have used the homogeneous differentiable group cohomology complex introduced in Definition \ref{defn:groupcoh-homogeneous}.
\end{definition}


Next, following the general strategy explained at the beginning of this section, we want to define the eta cocycle associated to $\varphi$.

\begin{definition}\label{sigma-cocycle}
Let $X$ be a closed manifold equipped with a proper, cocompact action of $G$, and let $\varphi\in C^k_{\rm diff}(G)$ be a smooth group 
cochain.  The {\em eta cochain} on ${}^b \mathcal{A}_{G,\RR}^c  ({\rm cyl}(X))$ associated to $\varphi$ is defined as
\begin{align*}
&\sigma^X_\varphi (B_0,\ldots,B_{k+1})\\
:=&\frac{(-1)^{k+1}}{2\pi}\int_{G^{k+1}}\int_\RR\\
&\hspace{1cm}{\rm Tr}\left(\hat{B}_0((g_1\cdots g_{k+1})^{-1},\lambda)\circ
\hat{B}_1(g_1,\lambda)\circ\cdots \circ\hat{B}_k(g_k,\lambda)\circ\frac{\partial\hat{B}_{k+1}(g_{k+1},\lambda)}{\partial \lambda}\right)\\
&\hspace{2cm}\varphi (e,g_1,g_1g_2,\ldots,g_1\cdots g_{k})
d\lambda dg_1\cdots dg_{k+1},
\end{align*}
where the notation in \eqref{mellin} has been used, and where we have used the homogeneous differentiable group cohomology complex introduced in Definition \ref{defn:groupcoh-homogeneous}.
\end{definition}
Using Lemma \ref{btrace} one can prove the following result \cite{PP2}:
\begin{proposition}
If $\varphi\in Z^{k}_{{\rm diff}} (G)$ is a smooth group {\bf cocycle} then $(\tau^{Y,r}_\varphi,\sigma^{\pa Y}_\varphi)$
is a relative cyclic cocycle for $ {}^b \mathcal{A}_{G}^c (Y)\xrightarrow{I} {}^b  \mathcal{A}_{G,\RR^+}^c  (\cyl(\pa Y)).$
\end{proposition}


\medskip
\noindent
In addition, one can prove the following:

\begin{proposition}
The pair 
 $(\tau^{Y,r}_\varphi,\sigma^{\pa Y}_\varphi)$ extends continuously to  a relative k-cocycle for 
$${}^b \mathcal{A}^\infty_G (Y)\xrightarrow{I} {}^b \mathcal{A}^\infty_{G,\RR} ({\rm cyl}(\partial Y))\,.$$
Moreover the following formula holds:
\begin{equation}\label{basic-formula-3--} 
 \langle  \tau^{Y}_\varphi,\Ind_\infty (D)\rangle=\langle  (\tau^{Y,r}_\varphi,\sigma^{\pa Y}_\varphi), \Ind_\infty (D,D_\partial )\rangle\,.
 \end{equation}

\end{proposition}

\begin{example} As explained in Example \ref{ex:sl2}, besides the trivial group cocycle, there is an interesting 
degree $2$ cocycle $A$ given by the area of a hyperbolic triangle in $\mathbb{H}$. The corresponding eta 3-cocycle on $\mathscr{S}(\RR,C^\infty_c(G))$ is given by
\begin{align*}
&\sigma_A(B_0,B_1,B_2,B_3)\\
:=&-\frac{1}{2\pi}\int_{G^3}\int_\RR\hat{B_0}((g_1g_2g_3)^{-1},\lambda)*\hat{B}_1(g_1,\lambda)*\hat{B}_2(g_2,\lambda)*\frac{\partial \hat{B}(g_3,\lambda)}{\partial\lambda}\\
&\hspace{2cm}{\rm Area}(\Delta_{\mathbb{H}}([e],g_1[e],g_1g_2[e]))d\lambda dg_1dg_2dg_3.
\end{align*}
\end{example}

\subsection{Higher APS index theorem associated to a group cocycle $\varphi\in Z^k_{{\rm diff}} (G)$}
Using \eqref{basic-formula-3--}, proceeding as we did for the delocalized trace ${\rm tr}^Y_g$  (and for the corresponding relative cocycle
$({\rm tr}^{Y,r}_g,\sigma^{\pa Y}_g)$), using the heat kernel approach to the Pflaum-Posthuma-Tang 
index formula developed in \cite{PP2},  one can establish the following higher Atiyah-Patodi-Singer index theorem

\begin{theorem}\label{maintheorem-PP-HP}
Let $Y_0$, $Y$ and $D$ as above. Assume that
the boundary operator $D_\partial$ is $L^2$-invertible. 
We consider $[\varphi]\in H^{2p}_{{\rm diff}}(G)$ and
$$\Ind_{\varphi}(D):= (-1)^p \frac{2p !}{p!}\langle \tau^Y_\varphi,  \Ind_\infty (D)\rangle.$$ Then
\begin{equation}\label{main-PP-HP}
\Ind_{\varphi}(D)=\int_Y c_Y {\rm AS}(Y) \wedge \Phi([\varphi])- \frac{1}{2}\eta_\varphi (D_\partial)\equiv  \int_{Y_0} c_{Y_0} {\rm AS}(Y_0) \wedge \Phi([\varphi]) - \frac{1}{2}\eta_\varphi (D_\partial) \,,
\end{equation}
where
 \begin{equation}
 \eta_{\varphi} (D_{\partial}):=
c_p\,
\left[ \sum_{i=0}^{2p}   \int_0^{\infty}\sigma^{\pa Y}_\varphi (p_t,\dots,[\dot{p}_t, p_t],\dots,p_t)dt
\right]
 \end{equation}
 with $p_t=V (t D_{{\rm cyl}})$   and $c_p=(-1)^p \frac{2p !}{p!}$.
\end{theorem}

For more details we refer to \cite{PP2}.

\subsection{Higher delocalized APS index theorem}
We finally come to the higher delocalized cyclic cocycles $\Phi^P_{Y,g}$, with $P$ a cuspidal parabolic subgroup with Langlands decomposition $MAN$
and $g$ a semisimple element in $M$.
The cyclic cocycle $\Phi^P_{Y,g}$ on $ {\mathcal{A}}^c_G(Y)$ can be written explicitly as
\[
\begin{aligned}
&\Phi^P_{Y,g}(A_0, A_1, \dots, A_m) \\
:=&\int_{h \in M/Z_M(g)} \int_{K N} \int_{G^{\times m}} C\big( H(g_1...g_mk) , H(g_2...g_mk ), \dots , H(g_mk) \big)\\
& \operatorname{Tr}\Big( A_0 \big (kh g h^{-1}nk^{-1} (g_1\dots g_m)^{-1}\big)\circ A_1(g_1) \dots \circ A_m(g_m) \Big) dg_1\cdots dg_m dk dn dh.
\end{aligned}
\]

Substituting the trace over $Y$ with the $b$-trace we define the cochain $\Phi^{P,r}_{Y,g}$ over  $ {^b \mathcal{A}}^c_G(Y)$.
We also have the corresponding eta cochain $\sigma^P_{\partial Y,g}$ on ${^b \mathcal{A}}^c_{G, \mathbb{R}}(\text{cyl}(Y))$:  for  $B_0,...,B_{m+1}\in {^b \mathcal{A}}^c_{G, \mathbb{R}}(\text{cyl}(Y))$
\[
\begin{aligned}
&\sigma^P_{\partial Y,g}(B_0, ..., B_{m+1})\\
:= &\int_{h \in M/Z_M(g)} \int_{K N} \int_{G^{\times {m+1}}} \int_{\mathbb{R}}C\big( H(g_1...g_mk) , H(g_2...g_mk ), \dots , H(g_mk) \big) \\
&\qquad {\operatorname{Tr}}\Big( \hat{B}_0 \big (kh g h^{-1}nk^{-1} (g_1\dots g_mg_{m+1})^{-1}, \lambda \big)\circ \hat{B}_1(g_1, \lambda)\circ \dots \circ \hat{B}_m(g_m, \lambda)\\
&\qquad\qquad \circ \frac{\partial \hat{B}_{m+1}(g_{m+1}, \lambda)}{\partial \lambda} \Big) dg_1\cdots dg_{m+1} dk dn dhd\lambda,
\end{aligned}
\]
where the notation in \eqref{mellin} has been used.
One proves that the pair $(\Phi^{r,P}_{Y,g}, \sigma^P_{\partial Y,g})$ defines a relative cyclic cocycle for the homomorphism
 $ {^b \mathcal{A}}^c_G(Y)\xrightarrow{I}   {}^b \mathcal{A}^c_{G,\RR}(\cyl(\partial Y))$.
Moreover:
\begin{itemize}
\item  the relative cyclic cocycle  $(\Phi^{r,P}_{Y,g}, \sigma_{\partial Y,g})$   extends 
continuously from the pair  $ {^b \mathcal{A}}^c_G(Y)\xrightarrow{I}   {}^b \mathcal{A}^c_{G,\RR}(\cyl(\partial Y))$
to the pair  $ {^b \mathcal{A}}^\infty_G(Y)\xrightarrow{I}   {}^b \mathcal{A}^\infty_{G,\RR}(\cyl(\partial Y))$.  
\item the crucial formula  $ \langle  \Phi_{Y,g}^P,\Ind_\infty (D)\rangle=\langle  (\Phi_{Y,g}^{P,r},,\sigma^P_{\pa Y,g}), \Ind_\infty (D,D_\partial )\rangle$ holds.
\end{itemize}
Using the last formula and proceeding as in the previous cases we arrive at the following result:\\
for each $s\in (0,1]$ 
\[
c_m \langle \Ind_\infty(D), [\Phi^P_{Y,g}]\rangle=
\Phi^{P,r}_{Y,g} (V(sD), \dots , V(sD))- \frac{1}{2}\int_s^\infty \eta^P_g(t)dt
\]
with 
$c_m=(-1)^{\frac{m}{2}}\frac{m!}{(\frac{m}{2})!}$ and $\eta^P_g (t):=2 c_m  \sum_{i=0}^{m} \sigma^P_{\partial Y,g} (p_t, ..., [\dot{p}_t, p_t], ..., p_t)$.
Part of the statement is of course that the $t$-integral converges at $+\infty$.

We would like to take the limit as $s\downarrow 0$; unfortunately we do not know how to compute the limit
of the first term on the right hand side (this should produce the local term in the index formula); in fact we cannot compute this limit even in the closed case.

\medskip

\noindent
{\bf Open problem.} Let $P<G$ a cuspidal parabolic subgroup with Langland's decomposition $P=MAN$ and let $g\in M$ be a semisimple element.
Let $D$ be a $G$-equivariant Dirac operator on a  closed cocompact $G$-proper manifold $X$; let $V(D)$ be the symmetrized Connes-Moscovici projector.
Can one prove that
$$\lim_{s\downarrow 0} \Phi^P_{X,g}(V(sD),\dots,V(sD))$$
exists ? If so, can one give an explicit formula for it ?

\medskip
\noindent
Because of these difficulties we go back to the proof of the higher delocalized index formula  in the closed case by reduction.
Thus we consider $Y_M:=Y/AN$, an $M$-proper manifold,  which has a slice decomposition  given by  $Y_M:=M\times_{K\cap M} Z$.
 The arguments of Hochs, Song and Tang in Theorem \ref{thm:reduction} can
be extended (with some efforts) to the case of manifolds with boundary yielding the following theorem (which is one of the main results in \cite{PPST}):

\begin{theorem}
\label{main thm}Suppose that the metric on $Y$ is slice-compatible. 
Assume that  $D_{\partial Y}$ is $L^2$-invertible and consider the higher 
index $\langle \Phi^P_{Y,g}, \Ind_\infty (D)\rangle$.
The following formula holds:
\[
\langle \Phi^P_{Y,g}, \Ind_\infty (D_Y)\rangle=
\int_{(Y_0/AN)^g} c^g_{(Y_0/AN)^g} {\rm AS}(Y_0/AN)_g-\frac{1}{2} \eta_g (D_{\partial Y_{M}}) 
\]
with 
$$\eta_g (D_{\partial Y_{M}}) =\frac{1}{\sqrt{\pi}} \int_0^\infty {\rm tr}^{\partial Y_{M}}_g (D_{\partial Y_{M}} \exp (-tD^2_{\partial Y_{M}}) \frac{dt}{\sqrt{t}}.$$
Here $c^g_{(Y_0/AN)^g}$ is a compactly supported smooth cutoff function on $(Y_0/AN)^g$ associated to the $Z_{M, g}$ action on $(Y_0/AN)^g$. 
\end{theorem}

\section{Geometric applications}\label{sec:application}
Now that we have explained various index theorems associated to a $G$-equi-variant Dirac operator on a $G$-proper manifold,
we discuss some geometric applications.

\subsection{Higher genera}
We begin by observing that  the homogeneous  space 
$G\slash K$ is a smooth model for $\underline{E}G$, the classifying space for proper actions of $G$, see \cite{BCH}: for any smooth proper action
of $G$ on a manifold $X$, there exists a smooth $G$-equivariant classifying map $\psi_X:X\to G\slash K$, unique up to $G$-equivariant homotopy.
For any proper action of $G$ on manifold  $X$ we consider $ \Omega^\bullet_{{\rm inv}} (X)$, the complex of $G$-invariant differential forms
on $X$ and its cohomology denoted by $H^\bullet_{\rm inv}(X)$. For a connected real reductive linear group $G$, we have the Van Est isomorphism:
$H^*_{{\rm diff}}(G)\simeq H^\bullet_{\rm inv}(G/K)$.
Consider now 
$[\alpha] \in H^\bullet_{{\rm inv}} (G/K)$ and let
 $\alpha\in\Omega^\bullet_{\rm inv}(G\slash K)$ a representative; consider its pull-back $\psi_X^*\alpha\in \Omega^\bullet_{\rm inv}(X)$ such that $[\psi_X^*\alpha]=\Phi([\alpha])$.
The higher signature associated to $[\alpha]$ is the real number
\begin{equation}\label{higher-s}
\sigma (X,[\alpha]) : = \int_X c_X L(X)\wedge \psi_X^* (\alpha),
\end{equation}
where $L(X)$ is the invariant de Rham form representing the $L$-class of $X$. 
The insertion of the cut-off function $c_X$, which as we know has compact support, ensures that the integral is well-defined (and it can be shown that it only depends on the class $ [L(X)\wedge \psi_X^* (\alpha)] \in H^\bullet_{\rm inv}(X)$).
The collection 
\begin{equation}\label{higher-s-collection}
\{\sigma (X,[\alpha]), [\alpha] \in H^\bullet_{{\rm inv}} (G/K)\}\end{equation}
are called the higher signatures of $X$; by the Van Est isomorphism they are labelled by the elements in $H^*_{{\rm diff}} (G)$.
Similarly, the higher $\widehat{A}$ genus associated to $X$ and to $[\alpha] \in H^\bullet_{{\rm inv}} (G/K)$
is the real number
\begin{equation}\label{higher-a}
\widehat{A}  (X,[\alpha]) : = \int_X c_X \widehat{A}(X)\wedge \psi_X^* (\alpha)\end{equation}
with $\widehat{A}(X)$ the de Rham class associated to the $\widehat{A}$-differential form for a $G$-invariant metric. The collection 
\begin{equation}\label{higher-a-collection}
\{\widehat{A}  (X,[\alpha]), \alpha \in H^\bullet_{{\rm inv}} (G/K)\}\end{equation}
are called the higher $\widehat{A}$-genera  of $Y$.\\
We have:

\begin{theorem}\label{theo:main}
Let $G$ be a connected real reductive Lie group. Let $X$ be an orientable manifold with a proper, cocompact action of $G$. Then the following holds true:
\begin{itemize}
\item[$(i)$] each higher signatures $\sigma (X,[\alpha])$, $[\alpha] \in H^\bullet_{{\rm inv}} (G/K)$,
is a $G$-homotopy invariant of $X$.
\item[$(ii)$] if $X$ admits a $G$-invariant Spin structure and a $G$-invariant metric of positive scalar curvature\footnote{from now on we shall briefly write PSC}
then each higher $\widehat{A}$-genus $\widehat{A}  (X,[\alpha])$, $[\alpha] \in H^\bullet_{{\rm inv}} (G/K)$,
vanishes.
\end{itemize}
\end{theorem}

The theorem follows easily from the higher index formula 
for differentiable group cocycles explained in the first part of this article and the usual stability properties of the $C^*$ index class of the signature operator and 
of the Spin-Dirac operator, established in this context by Fukumoto \cite{Fukumoto} and Guo-Mathai-Wang \cite{guo-mathai-wang-psc}. 
As we have already explained, the theorem in fact holds more generally for  $G$ a Lie group with finitely many connected components satisfying property RD, and such that $G/K$
is of non-positive sectional curvature for a maximal compact subgroup $K$. 
See \cite{PP1} for more details.

The corresponding APS index theorems can be used to introduce relative higher genera
$\sigma (X,\partial X, [\alpha]) $ and $\widehat{A}  (X,\partial X, [\alpha])$ and prove, for example, additivity results for closed manifolds
that are obtained by gluing manifolds with boundary  along $G$-diffeomorphic boundaries. In addition, the relative
higher $\widehat{A}$-genera can be used to produce obstructions to the existence of an isotopy from a $G$-invariant PSC metric on $\partial X$ to a
$G$-invariant metric on $\pa X$ which extends to a $G$-invariant metric which is PSC on all of $X$.

\subsection{Rho invariants}

In this subsection we shall briefly introduce (higher) rho numbers associated to positive scalar curvature  metrics and
$G$-equivariant homotopy equivalences.
All our Dirac operators will be  $L^2$-invertible; indeed, if we want
to consider bordism properties of these numbers we do need $L^2$-invertibility so as to be able to define
an  APS index class on the manifold with boundary realising the bordism.

\medskip
\noindent
{\bf Rho numbers associated to delocalized 0-cocycles.}
We consider a closed $G$-proper manifold $X$ {\bf without} boundary, $G$ connected, linear real reductive, $g\in G$ a semisimple element,
$D_X$ a $G$-equivariant $L^2$-invertible Dirac operator of the form (\ref{dec dirac}). We consider 
$$\eta_g (D_X):=\frac{1}{\sqrt{\pi}} \int_0^\infty {\rm tr}^X_g (D_{X} \exp (-tD^2_{X}) \frac{dt}{\sqrt{t}}.$$
 Let $X$ be $G$-equivariantly Spin  and $D_X\equiv D_{\mathbf{h}}$, the Spin Dirac operator
 associated to a $G$-equivariant PSC metric $\mathbf{h}$. Then we know that $D_\mathbf{h}$ is $L^2$-invertible. We define 
$$\rho_g (\mathbf{h}):= \eta_g (D_{\mathbf{h}})\,.$$
If, on the other hand, $\mathbf{f}:X_1\to X_2$ is a $G$-homotopy equivalence, then using \cite{Fukumoto}
we know that there exists a bounded perturbation $B_\mathbf{f}$ of the signature operator on $X:= X_1\sqcup (-X_2)$, where $-X_2$ is the same manifold as $X_2$ with the opposite orientation,  that
makes it invertible. Moreover, one can prove  that this perturbation $B_{\mathbf{f}}$ is in 
$C^* (X,\Lambda^*)^G$. Hence, by density, we conclude that there exists a perturbation 
$B^\infty _{\mathbf{f}}\in \mathcal{A}^\infty_G (X,\Lambda^*)$ such that $D^{{\rm sign}}_X+B^\infty_{\mathbf {f}}$ is $L^2$-invertible.
It is possible to extend  Theorem \ref{thm etadefine} to this perturbed situation  and define the rho-number of the homotopy equivalence $\mathbf {f}$ as
$$\rho_g (\mathbf{f}):= \eta_g  (D^{{\rm sign}}_X+B^\infty_\mathbf{f})\,.$$
We refer to \cite{PPST2} for a detailed discussion of the index theory associated to perturbed operators such as the one appearing above.

\medskip
\noindent
{\bf Rho numbers associated to higher delocalized cocycles}
We can generalize the above definitions and define rho numbers associated
 the higher cocycles $\Phi^P_{g}$. More precisely,
let $P=MNA$, $g\in M$ as above and consider $\Phi_g^P$ and $\Phi_{X,g}^P$ (we recall that $X$ is without boundary). 
Assume that $\mathbf{h}$ is a slice-compatible 
$G$-invariant PSC metric on $X$. 
Then  
\begin{equation}\label{rho-P}
\rho_g^P (\mathbf{h}):= \eta_g (D_{X_M})
\end{equation}
(with $X_M$ the reduced manifold associated to $X$)
 is well defined. Notice that it is proved \cite{PPST} that if  $D_X$  invertible, then $D_{X_M}$ is also invertible.

\medskip
\noindent
{\bf Bordism properties.}
The APS index theorems introduced in this article can be used to study the bordism properties 
of these rho invariants.
We concentrate on the case of psc metrics.
We assume, unless otherwise stated, that we are on a $G$-proper manifold which is endowed 
with a slice-compatible 
$G$-invariant metric and a slice-compatible $G$-invariant Spin structure.  

Let $(Y,\mathbf{h}_0)$ and $(Y,\mathbf{h}_1)$ be two slice-compatible  
psc metrics. We say that they are $G$-concordant if there exists a $G$-invariant metric
$\mathbf{h}$ on $Y\times [0,1]$ 
 which is of PSC, product-type near the boundary and restricts to $\mathbf{h}_0$ at $Y\times \{0\}$ and to
$\mathbf{h}_1$ at $Y\times \{1\}$.

The following Proposition is an example of the applications one can  envisage for these secondary invariants. Before stating it,
we remark  that if $g$ is non-elliptic, that is, does not conjugate to a compact element, then every element of the conjugacy class $C(g):=\{hgh^{-1}|h\in G\}$ in $G$ does not have any fixed point on a $G$-proper manifold. 


 \begin{proposition}\label{prop:bordism}$\;$\\
 1] Assume that the $G$-invariant 
 psc metrics  $\mathbf{h}_0$ and $\mathbf{h}_1$ on $Y$  are $G$-concordant.
 Assume that $g$ is non-elliptic on $Y$. Then $\rho_g (\mathbf{h}_0)=\rho_g (\mathbf{h}_1)$.\\
 2] Let $P=MAN<G$ be a cuspidal parabolic subgroup and let $x\in M$ be a semisimple element.
 Assume that the $G$-invariant slice-compatible
 psc metrics  $\mathbf{h}_0$ and $\mathbf{h}_1$ on $Y$  are $G$-concordant.
 Assume that $g$ is non-elliptic  on $Y_M$. Then $\rho^P_g (\mathbf{h}_0)=\rho^P_g (\mathbf{h}_1)$.\\
  Put it differently,   for such $g$ 
  our (higher) rho invariants are in fact {\it concordance invariants}.
  \end{proposition}

\begin{proof}
In both cases the proof is an immediate consequence of  the relevant delocalized APS index theorems.
\end{proof}

\begin{remark}
In the case of free proper actions of discrete groups, higher rho invariants
have been employed very successfully in studying the moduli space of concordant metrics 
of positive scalar curvature. See \cite{Xie-Yu-moduli} and \cite{PSZ}.
It is a challenge to understand whether the results we have just explained in the context of proper actions of Lie groups can be employed in studying the space
$ \mathcal{R}^+_{G,{\rm slice}} (Y)$ of slice-compatible metrics of PSC (if non-empty)
or the space $\mathcal{R}^+_G (Y)$ (arbitrary $G$-equivariant metrics of PSC).
For these questions it would be interesting to develop a $G$-equivariant Stolz' sequence and  investigate
 its basic properties. 
 We leave this task to future research.
\end{remark}

\bibliographystyle{amsalpha}

\end{document}